\documentclass[a4paper,11pt]{amsart}
\usepackage{amscd,amssymb,amsmath,mathrsfs,latexsym,url}
\usepackage{hyperref}
\usepackage[mac]{inputenc}
\usepackage[T1]{fontenc}
\usepackage[french]{babel}
\hypersetup{breaklinks=true}

\catcode`\;=\active
\def;{\relax\ifhmode\ifdim\lastskip>\z@
\unskip\fi\kern.2em\fi\string;}

\catcode`\!=\active
\def!{\relax\ifhmode\ifdim\lastskip>\z@
\unskip\fi\kern.2em\fi\string!}

\catcode`\é=\active \def é{\'e}
\catcode`\à=\active \def à{\`a}
\catcode`\è=\active \def è{\`e}
\catcode`\ù=\active \def ù{\`u}
\catcode`\â=\active \def â{\^a}
\catcode`\ê=\active \def ê{\^e}
\catcode`\î=\active \def î{\^i}
\catcode`\ï=\active \def ï{\"i}
\catcode`\ô=\active \def ô{\^o}
\catcode`\û=\active \def û{\^u}
\catcode`\ä=\active \def ä{\"a}
\catcode`\ë=\active \def ë{\"e}
\catcode`\ö=\active \def ö{\"o}
\catcode`\ü=\active \def ü{\"u}
\catcode`\ç=\active \def ç{\c c}
\def\pre#1{$#1$\hbox{\hskip -\mathsurround -}}

\font\tensc=cmcsc10	scaled 1000

\def\UU#1{{\ifmmode\underline{#1}\else\underbar{#1}\fi}}
\def\OO#1{{\ifmmode\overline{#1}
     \else$\setbox0=\hbox{#1} \dp0=0pt\mathsurround=0pt \overline{\box0}$\fi}}
\def\QQ#1{{\ifmmode{\boldsymbol{#1}}\else{\bf#1}\fi}}
\def\RQ#1{{\ifmmode{\mathbf{#1}}\else{\bf#1}\fi}}
\def\RR#1{{\rm #1}}   
\def\SS#1{{\mathcal #1}}   
\def\AA#1{{\mathfrak #1}}  

\numberwithin{equation}{section}
\newtheorem{thm}[equation]{Théorème}
\newtheorem{prop}[equation]{Proposition}
\newtheorem{cor}[equation]{Corollaire}
\newtheorem{lem}[equation]{Lemme}
\newtheorem{definition}[equation]{D\'efinition}

\newtheorem{exes}[equation]{Exemples}

\newenvironment{undef}[1]{\noindent{\bf #1.} }{\vspace{3.3mm}}

\title[Groupes de Galois et nombres automatiques]{Groupes de Galois et nombres automatiques}

\author[Patrice Philippon]{Patrice Philippon}

\address{Institut de Math{\'e}matiques de Jussieu -- U.M.R. 7586 du CNRS, \'Equipe de Th\'eorie des Nombres. BP 247, 4 place Jussieu, 75005 Paris, France.}

\email{patrice.philippon@imj-prg.fr}

\urladdr{\url{http://www.math.jussieu.fr/~pph}}



\begin{document}

\begin{abstract}
Nous montrons dans le cadre de la théorie de Mahler, que les relations de dépendance algébrique sur $\OO{\RQ Q}$ entre les valeurs de fonctions solutions d'un système d'équations fonctionnelles proviennent, par spécialisation, des relations entre les fonctions elle-mêmes. Nous en déduisons quelques résultats nouveaux sur l'indépendance linéaire des valeurs de fonctions \pre{q}régu\-liè\-res.

\medskip

\centerline{\tensc Abstract}
\noindent{\bf Galois groups and automatic numbers}: In the frame of Mahler's method for algebraic independence, we show that the algebraic relations over $\OO{\RQ Q}$ linking the values of functions solutions of a system of functional equations come from the algebraic relations between the functions themselves, by specialisation. We deduce some new results on the linear independence of values of \pre{q}regular functions.
\end{abstract}

\maketitle

\tableofcontents

\section{Introduction}\label{sec:introduction}
Les relations de dépendance algébrique entre valeurs de fonctions analytiques ou méro\-mor\-phes font encore l'objet d'un large mystère. Elles héritent bien évi\-dem\-ment, par spécialisation, des relations liant les fonctions elles-mêmes et souvent se cantonnent à cet héritage. C'est par exemple le cas dans la théorie des \pre{E}fonctions et plus généralement des séries Gevrey, tout au moins en dehors d'un éventuel ensemble fini d'arguments exceptionnels.

Plus précisément, dans le cadre des \pre{E}fonctions en une variable complexe $z$ le théorème de Siegel-Shidlovskii énonce que pour les arguments algébriques en dehors des points singuliers du système différentiel satisfait par une famille de \pre{E}fonctions, le degré de transcendance sur $\RQ Q$ des valeurs de ces \pre{E}fonctions est égal au degré de transcendance sur $\RQ Q(z)$ des fonctions.

Cet énoncé remarquable ne donne toutefois pas directement la nature exacte des relations liant les valeurs en question. Par exemple, ces valeurs sont-elles linéairement indépendantes sur $\RQ Q$ dès que les fonctions sont liné\-ai\-re\-ment indépendantes sur $\RQ Q(z)$? Ou bien, la valeur d'une fonction transcendante sur $\RQ Q(z)$ de la famille, est-elle transcendante sur $\RQ Q$?

Cet aspect a été étudié par exemple dans~\cite{NS=96} où Y.V.Nesterenko et A.B.Shidlovskii ont montré qu'en dehors d'un ensemble fini effectivement calculable (contenant les singularités du système différentiel impliqué) toute relation algébrique entre les valeurs provient de la spécialisation d'une relation algébrique entre les fonctions. Dans~\cite{Be=06} F.Beukers, s'appuyant sur l'approche radicalement nouvelle d'Y.André~\cite{An=00}, montre que l'ensemble exceptionnel est en fait réduit aux seules singularités du système différentiel. Et récemment, Y.André~\cite{An=12} remarque que ce dernier énoncé se déduit aussi du théorème de Siegel-Shidlovskii, indépendamment de la méthode de~\cite{An=00}, via la théorie de Galois différentielle.

Notons toutefois que la connaissance des relations algébriques liant toutes les fonctions est nécessaire pour en déduire la transcendance de la valeur d'une seule de ces fonctions, par exemple. En effet, par spécialisation une relation de dépendance algébrique impliquant plusieurs fonctions peut se réduire à une relation de dépendance algébrique de la valeur d'une seule d'entre elles. L'idéal des relations algébriques entre les fonctions peut être déterminé par la connaissance du groupe de Galois du système différentiel satisfait par les \pre{E}fonctions.

\medskip

Notre propos ici est d'une part d'obtenir des théorèmes semblables dans le cadre de la théorie de Mahler. Et d'autre part de préciser ces résultats dans le cas des séries automatiques et plus généralement des séries \pre{q}régulières, complétant ainsi les résultats de~\cite{Bk=94}. La méthode de Mahler conduit à un énoncé tout-à-fait similaire au théorème de Siegel-Shidlovskii, obtenu par Ku.Nishioka.

\begin{thm}{\rm (\cite[Theorem 4.2.1]{Ni=96})}\label{thm:nishioka}
Soit $N\geq 1$, $q\geq 2$ des entiers, $\rho>0$ un réel positif, $\QQ f(z)$ un vecteur colonne de $N$ séries $f_1(z),\dots,f_N(z)$ en $z$ à coefficients dans un corps de nombres, convergeant dans le disque sans bord de rayon $\rho$ centré en l'origine dans $\RQ C$ et satisfaisant le système d'équations fonctionnelles
$$\QQ f(z) = A(z)\QQ f(z^q)
\enspace,$$
où $A(z)$ est une matrice $N\times N$ inversible, à coefficients dans $\OO{\RQ Q}(z)$.

Soit $\alpha$ un nombre algébrique non nul, de valeur absolue $<\rho$ tel qu'aucune de ses puissances $\alpha^{q^\ell}$, $\ell\in\RQ N$, ne soit pôle d'un coefficient de $A(z)^{-1}$, alors
$$\RR{degtr}_{\OO{\RQ Q}}\OO{\RQ Q}(f_1(\alpha),\dots,f_N(\alpha)) \geq \RR{degtr}_{\OO{\RQ Q}(z)}\OO{\RQ Q}(z,f_1(z),\dots,f_N(z))
\enspace.$$ 
\end{thm}

P-G.Becker~\cite{Bk=94} remarque que ce théorème~\ref{thm:nishioka}, à l'instar du théorème de Siegel-Shidlovskii, n'entraîne pas la transcendance des valeurs de fonctions automatiques transcendantes aux points algébriques du disque de convergence. Pourtant, utilisant une approche complètement différente, via le théorème du sous-espace de Schmidt, B.Adamczewski et Y.Bugeaud~\cite{AdBu=07} ont montré qu'un nombre automatique est transcendant ou rationnel, cette dernière éventualité intervenant lorsque la série automatique associée est une fonction rationnelle. Mais, ce résultat ne découle pas directement du théorème~\ref{thm:nishioka}. On sait par ailleurs qu'une fonction automatique, ou plus généralement \pre{q}régulière, algébrique est nécessairement rationnelle, \emph{cf.}~\cite[Lemma 5]{Bk=94}.

\medskip

Soit $q\geq 2$ un entier et $A(z)$ une matrice $N\times N$ inversible, à coefficients dans $\OO{\RQ Q}(z)$. Nous supposons qu'il existe un vecteur $\QQ f(z)$, à composantes dans $\OO{\RQ Q}[[z]]$, solution du système d'équations fonctionnelles
\begin{equation}\label{eq:systintro}
\QQ f(z)=A(z)\QQ f(z^q)
\enspace.\end{equation}
On sait que les composantes de $\QQ f(z)$ définissent des fonctions méromorphes dans le disque unité sans bord de $\RQ C$, \emph{voir} le lemme~\ref{lem:convergence} de la section~\ref{sec:convergence}. Nous allons montrer que les relations de dépendance algébrique sur $\OO{\RQ Q}$ entre les valeurs des composantes de $\QQ f(z)$ (sauf en un ensemble d'arguments particuliers) proviennent, par spécialisation, des relations de dépendance algébrique sur $\OO{\RQ Q}[z]$ entre les composantes elles-mêmes.

On note $T(z)=T_A(z)\in\OO{\RQ Q}[z]$ le polynôme de plus petit degré satisfaisant $T(0)=1$ et dont les zéros sont les pôles des coefficients des matrices $A(z)$ et $A(z)^{-1}$, distincts de $0$, comptés avec multiplicité. L'union des racines des polynômes $T(z^{q^\ell})$, $\ell\in\RQ N$, sera alors appelée l'\emph{ensemble singulier} (et ses éléments \emph{singularités}) de~\eqref{eq:systintro}. Nous observons qu'avec cette définition tout pôle d'une composante de $\QQ f(z)$ est singularité de~\eqref{eq:systintro}.
\begin{thm}\label{thm:indalg}
Soit $\alpha\in\OO{\RQ Q}\subset\RQ C$, $0<|\alpha|<1$, tel que $T(\alpha^{q^\ell})\not=0$ pour tout $\ell\in\RQ N$, alors pour tout $P\in\OO{\RQ Q}[X]$ tel que $P(\QQ f(\alpha))=0$, il existe $Q\in\OO{\RQ Q}[z,X]$ de même degré en $X$ que $P$, tel que $Q(z,\QQ f(z))=0$ et $P(X)=Q(\alpha,X)$.
\end{thm}
Ce résultat est l'analogue fonctionnel de celui établi par F.Beukers dans~\cite[Theorem 1.3]{Be=06} dans le cas différentiel des \pre{E}fonctions. La démonstration que nous en donnons à la section~\ref{sec:regspe} est toutefois plutôt à rapprocher de celle proposée par Y.André dans~\cite[Corollary 1.7.1]{An=12}. Pour élucider les relations de dépendance algébriques liant les composantes de $\QQ f(\alpha)$, il faut donc encore trouver des générateurs de l'idéal des relations entre les composantes de $\QQ f(z)$ sur $\OO{\RQ Q}(z)$. Et pour cela la connaissance du groupe de Galois associé au système~\eqref{eq:systintro} peut être une étape cruciale. Nous rappelons aux sections~\ref{sec:corpsendo} et~\ref{sec:Galois} les éléments de théorie des corps à endomorphisme et de théorie de Galois pertinents dans ce contexte.

\medskip

On obtient comme conséquence directe du théorème~\ref{thm:indalg} le suivant, qui recopie le corollaire~\ref{cor:thmindlinautnondeg} démontré à la section~\ref{sec:independancelin}.

\begin{thm}\label{thm:indlinaut}
Soit $\RQ k_0\subset\RQ C$ un corps de nombres, $\alpha\in\RQ k_0$, $0<|\alpha|<1$, tel que $T(\alpha^{q^h})\not=0$ pour tout $h\in\RQ N$ et $1\le\ell\le N$. Supposons que la matrice $A(z)$ du système~\eqref{eq:systintro} soit à coefficients dans $\RQ k_0(z)$ et le vecteur $\QQ f(z)$ à composantes $f_1(z),\dots,f_N(z)$ dans $\RQ k_0[[z]]$. Si les nombres $f_{1}(\alpha),\dots,f_{\ell}(\alpha)$ sont linéairement indépendants sur $\RQ k_0$ alors ils sont linéairement indépendants sur $\OO{\RQ Q}$.
\end{thm}

Nous déduisons du théorème~\ref{thm:indlinaut} le théorème~\ref{thm:indlinautsing} suivant pour les systèmes à coefficients polynomiaux (en particulier pour les fonctions \pre{q}régulières étudiées dans~\cite{Bk=94}, \emph{voir} aussi~\cite[Theorem 5.1.2]{Ni=96}). Dans ce cas, les composantes dans $\OO{\RQ Q}[[z]]$ du vecteur solution $\QQ f(z)$ de~\eqref{eq:systintro} convergent dans le disque unité sans bord de $\RQ C$ et y définissent des fonctions analytiques, \emph{voir} lemme~\ref{lem:convergence} à la section~\ref{sec:convergence}. L'enjeux est de réintégrer les singularités du système d'équations fonctionnelles dans l'énoncé.

La preuve consiste à se ramener, lorsque $\alpha$ est une singularité, à un système nettoyé de cette singularité. Nous utilisons à cette fin le lemme ad hoc~\ref{lem:desingaut}, mais nous remarquons que B.Adamczewski et C.Faverjon disposent d'un lemme de réduction des singularités~\cite[lemme~6]{AdFa=15}, en fait plus simple et plus efficace que notre énoncé. C'est d'ailleurs après que B.Adamczewski m'ait mentionné ce résultat que j'ai pu conclure la preuve du lemme~~\ref{lem:desingaut}. Le passage du théorème~\ref{thm:indlinaut} au théorème~\ref{thm:indlinautsing} leur revient donc.

\begin{thm}\label{thm:indlinautsing}
Soit $\RQ k_0\subset\RQ C$ un corps de nombres, $\alpha\in\RQ k_0$, $0<|\alpha|<1$, et $1\le\ell\le N$. Supposons que la matrice $A(z)$ du système~\eqref{eq:systintro} soit à coefficients dans $\RQ k_0[z]$ et le vecteur $\QQ f(z)$ à composantes $f_1(z),\dots,f_N(z)$ dans $\RQ k_0[[z]]$. Si les nombres $f_{1}(\alpha),\dots,f_{\ell}(\alpha)$ sont linéairement indépendants sur $\RQ k_0$ alors ils sont linéairement indépendants sur $\OO{\RQ Q}$.
\end{thm}
La démonstration de cet énoncé sera donnée avec le corollaire~\ref{cor:corlinindaut}. Il serait intéressant de l'étendre à tous les systèmes à coefficients dans $\RQ k_0(z)$ (et non seulement $\RQ k_0[z]$). De même, il serait intéressant de développer la théorie pour des transformations rationnelles plus générales, comme dans~\cite[Chapitre 3]{Zo=10} par exemple.

Soit $q,b\ge 2$ des entiers, un nombre \pre{(q,b)}automatique est un nombre dont la suite des chiffres dans son écriture en base $b$ est \pre{q}automatique, \emph{voir}~\cite[Chapters 5 \& 13]{AS=03}. C'est donc la valeur en $1/b$ de la série génératrice d'une suite automatique, qui est un cas particulier de fonction \pre{q}régulière, analytique dans le disque unité sans bord de $\RQ C$. On obtient donc comme corollaire du théorème~\ref{thm:indlinautsing} la généralisation suivante du théorème de B.Adamczewski et Y.Bu\-geaud~\cite{AdBu=07}.

\begin{cor}\label{cor:AdamBug}
Des nombres \pre{(q,b)}automatiques linéairement indépendants sur $\RQ Q$ sont linéairement indépendants sur $\OO{\RQ Q}$. 
\end{cor}

La transcendance des nombres automatiques irrationnels en découle en ajoutant au besoin la fonction $1$ comme composante supplémentaire au vecteur $\QQ f(z)$ solution du système d'équations fonctionnelles~\eqref{eq:systintro}. Le nouveau système à considérer s'écrit alors
$$\left(\begin{matrix}1\\ \QQ f(z)\end{matrix}\right) = \left(\begin{matrix}1 &0\\ 0 &A(z)\end{matrix}\right)\left(\begin{matrix}1\\ \QQ f(z^q)\end{matrix}\right)
\enspace.$$ 
Ceci fournit donc une démonstration alternative à celle de~\cite{AdBu=07}.

\medskip

Lorsque le système~\eqref{eq:systintro} (avec $A(z)$ à coefficients fonctions rationnelles) admet une matrice fondamentale $U(z)$ (et non seulement un vecteur $\QQ f(z)$) solution à coefficients analytiques dans le disque unité sans bord de $\RQ C$, on a le résultat direct suivant.
\begin{thm}\label{thm:indlin}
Sous les hypothèses ci-dessus, soit $\QQ f(z)$ un vecteur solution de~\eqref{eq:systintro} dont les composantes $f_1(z),\dots,f_N(z)$ sont dans $\OO{\RQ Q}(z,U(z))$ et linéairement indépendantes sur $\OO{\RQ Q}(z)$.

Soit $\alpha\in\OO{\RQ Q}\subset\RQ C$, $0<|\alpha|<1$, tel que $T(\alpha^{q^\ell})\not=0$ pour $\ell\in\RQ N$, alors les $N$ nombres $f_1(\alpha),\dots, f_N(\alpha)$ sont linéairement indépendants sur $\OO{\RQ Q}$.
\end{thm}
C'est le corollaire~\ref{cor:indlin}, dont la démonstration est donnée à la section~\ref{sec:independancelin}. On y montre également par les exemples~\ref{exe:irreginfini} que toutes les hypothèses du théorème~\ref{thm:indlin} sont nécessaires.

Enfin, nous établissons avec le corollaire~\ref{cor:enonceconvergence} de la section~\ref{sec:convergence} une condition nécessaire et suffisante à l'existence d'une matrice fondamentale de solutions de~\eqref{eq:systintro} à coefficients méromorphes dans le disque unité sans bord de $\RQ C$.

\section{Corps à endomorphisme}\label{sec:corpsendo}
Un \emph{\pre{\sigma}anneau} est un anneau commutatif, unitaire, $\AA A$ muni d'un endomorphisme unitaire (\emph{i.e.} satisfaisant $\sigma(1)=1$). Un \emph{\pre{\sigma}corps} est un \pre{\sigma}anneau $(\AA A,\sigma)$ tel que $\AA A$ soit un corps. Les éléments $y$ d'un \pre{\sigma}anneau (\emph{resp.} \pre{\sigma}corps) satisfaisant $\sigma(y)=y$ sont appelés les \emph{\pre{\sigma}constantes}, ils forment un sous-anneau (\emph{resp.} sous-corps), \emph{voir}~\cite[Définition 1.1.1]{NPQ=12}. Un homomorphisme de \pre{\sigma}anneau (\emph{resp.} \pre{\sigma}corps) est un homomorphisme d'anneau (\emph{resp.} corps) qui commute aux \pre{\sigma}actions.

Soit $q\ge2$ un entier, $\RQ k\subset\RQ C$ un sous-corps algébriquement clos dans $\RQ C$ et $\sigma$ l'endomorphisme (unitaire) du corps $\RQ k(z)$ défini par $\sigma(a(z))=a(z^q)$, de sorte que $(\RQ k(z),\sigma)$ est un \pre{\sigma}corps, de corps des \pre{\sigma}constantes égal à $\RQ k$. Avec ces notations le système d'équations fonctionnelles~\eqref{eq:systintro} s'écrit
$$\QQ f = A\sigma(\QQ f)\quad\mbox{ou}\quad \sigma(\QQ f) = A^{-1}\QQ f
\enspace,$$
avec $A$ une matrice inversible $N\times N$ à coefficients dans $\RQ k(z)$.

\begin{definition}\label{def:matricefondam}
Soit $(\AA A,\sigma)$ un \pre{\sigma}anneau. Une \emph{matrice fondamentale} de solutions dans $\AA A$ du système d'équations fonctionnelles~\eqref{eq:systintro} est une matrice $U$ inversible, $N\times N$ à coefficients dans $\AA A$ satisfaisant $U=A\sigma(U)$.
\end{definition}

Nous considérons maintenant un \pre{\sigma}corps $\AA K$ contenant $\RQ k(z)$ et tous les coefficients d'une matrice fondamentale de solutions de~\eqref{eq:systintro}. Et nous montrons que tout vecteur solution de~\eqref{eq:systintro} dans $\AA K^N$ est une combinaison linéaire à coefficients dans $\RQ k$ des vecteurs colonnes de la matrice fondamentale.

\begin{prop}\label{wronskien}
Soit $\AA K$ un \pre{\sigma}corps, de corps des \pre{\sigma}constantes égal à ${\RQ k}$ et contenant les coefficients d'une matrice fondamentale de solutions $U$. Alors, tout vecteur $\QQ y$, de composantes dans $\AA K$, solution de~\eqref{eq:systintro}, est combinaison linéaire à coefficients dans ${\RQ k}$, des colonnes de $U$.
\end{prop}
\begin{proof}
Le vecteur $\QQ y$ et les colonnes de $U$ appartiennent à $\AA K^N$ et sont donc liés par une relation linéaire non triviale à coefficients dans $\AA K$, soit $\QQ yb_0=\sum_{j=1}^N\QQ u_{j}b_{j}$ avec $b_0,b_{j}\in\AA K$. Comme les vecteurs $\QQ u_{j}$ sont indépendants ($\det(U)\not=0$) on a $b_0\not=0$ et, quitte à diviser par $b_0$, on peut supposer $b_0=1$:
\begin{equation}\label{eqdanspreuvewronskien}
\QQ y = \sum_{j=1}^N \QQ u_{j}b_{j}
\enspace.\end{equation}
En faisant agir $\sigma$ on obtient $\sigma(\QQ y)=\sum_{j=1}^N\sigma(\QQ u_{j})\sigma(b_{j})$, puis en multipliant à gauche par $A$ il vient $\QQ y=\sum_{j=1}^N\QQ u_{j}\sigma(b_{j})$, car $A\sigma(\QQ y) = \QQ y$ et $A\sigma(\QQ u_{j}) = \QQ u_{j}$ par~\eqref{eq:systintro}. En retranchant~\eqref{eqdanspreuvewronskien} on obtient finalement
$$\sum_{j=1}^N\QQ u_{j}(\sigma(b_{j})-b_{j}) = 0$$
qui, vu l'indépendance des vecteurs $\QQ u_{j}$, implique $\sigma(b_{j})=b_{j}$ pour tout $j=1,\dots,N$. Enfin, le corps des \pre{\sigma}constantes de $\AA K$ étant ${\RQ k}$ cela montre $b_{j}\in{\RQ k}$, $j=1,\dots,N$, dans~\eqref{eqdanspreuvewronskien}, et $\QQ y$ est bien combinaison linéaire à coefficients dans ${\RQ k}$ des vecteurs colonnes $\QQ u_{j}$ de la matrice $U$.
\end{proof}

On note $K:=\bigcup_{\ell\in\RQ N}\RQ k(z^{q^{-\ell}})$ la \emph{clôture inversive} de $\RQ k(z)$ que nous munissons de l'automorphisme, encore noté $\sigma$, étendant l'endomorphisme $\sigma$ de $\RQ k(z)$. 

Soit $1\le r\le N$ un entier et $V=(u_{i,j})_{\stackrel{1\le i\le N}{\scriptscriptstyle 1\le j\le r}}$ la matrice des composantes de $r$ vecteurs $\QQ u_1,\dots,\QQ u_r$ solutions de~\eqref{eq:systintro}. L'action de l'endomorphisme $\sigma$ s'étend naturellement au corps $K(V)$ en posant $\sigma(V(z))=V(z^q)$.

Soit $X:=\big(X_{i,j}\big)_{\stackrel{1\leq i\leq N}{\scriptscriptstyle 1\leq j\leq r}}$ une matrice de variables indé\-pen\-dan\-tes. Notons $K[X]$ l'anneau des polynômes en les variables $X_{i,j}$ à coefficients dans $K$ et $\AA P$ l'idéal premier de $K[X]$ des relations de dépendance algébrique sur $K$ entre les fonctions $(u_{i,j})_{\stackrel{1\leq i\leq N}{\scriptscriptstyle 1\leq j\leq r}}$. Introduisons les automorphismes de \pre{K}algèbres, inverses l'un de l'autre,
$$\begin{aligned}
\tilde\sigma:&\kern8pt K[X]&\longrightarrow&\kern10pt K[X]\\
&\kern18pt z&\longmapsto&\kern17pt z^{q}\\
&\kern16pt X&\longmapsto&\kern2pt A(z)^{-1}X
\end{aligned}
\hskip45pt
\begin{aligned}
\tilde\sigma^{-1}:&\kern8pt K[X]&\longrightarrow&\kern10pt K[X]\\
&\kern18pt z&\longmapsto&\kern13pt z^{1/q}\\
&\kern16pt X&\longmapsto&\kern2pt A(z^{1/q})X
\end{aligned}\enspace.$$
Ceci fait de $K[X]$ une \pre{\tilde\sigma}algèbre sur le \pre{\sigma}corps $K$. Pour $P\in K[X]$ on vérifie
$$\tilde\sigma(P)(V)=\sigma(P(V))\quad\mbox{et}\quad
P(V)=\sigma(\tilde\sigma^{-1}(P)(V))
\enspace,$$
car $V=A\sigma(V)$. En particulier, $\tilde\sigma(\AA P)=\AA P$ et l'action de $\tilde\sigma$ coïncide avec celle de $\sigma$ sur $K[V]$ après spécialisation de $X$ en $V$. Ainsi, la \pre{\sigma}algèbre $K[V]$ est isomorphe à la \pre{\tilde\sigma}algèbre $K[X]/\AA P$.

\section{Groupes de Galois}\label{sec:Galois}
Dans cette section on se place sur un sous-corps $\RQ k$ algé\-bri\-que\-ment clos de $\RQ C$.
On notera alors $K:=\bigcup_{\ell\in\RQ N}\RQ k(z^{q^{-\ell}})$ la clôture inversive de $\RQ k(z)$, telle qu'introduite à la fin de la section~\ref{sec:corpsendo} précédente.

Reprenons le système d'équations fonctionnelles~\eqref{eq:systintro} et une matrice fondamentale de solutions $U=U(z)$. Posons $r=N$ dans les notations introduites à la fin de la section~\ref{sec:corpsendo}. On a alors la matrice $X:=\big(X_{i,j}\big)_{1\leq i,j\leq N}$ de variables indépendantes et la \pre{K}algèbre $K[X]$ sur laquelle $\tilde\sigma$ agit par $\tilde\sigma(X)=A(z)^{-1}X$ et $\tilde\sigma(z)=z^q$, c'est un \pre{\tilde\sigma}anneau. Comme $U$ est une matrice fondamentale de solutions de~\eqref{eq:systintro} l'idéal premier $\AA P=(P\in K[X];P(z,U(z))=0)$ est un \pre{\tilde\sigma}idéal, {\it i.e.} $\tilde\sigma(\AA P)\subset\AA P$,~\cite[Définition 2.1.1]{NPQ=12}. L'action de $\tilde\sigma$ sur la \pre{K}algèbre intègre
$$K[X]/\AA P \simeq K[U]
\enspace,$$
déduite de celle sur $K[X]$, coïncide avec l'extension naturelle de $\sigma$ à $K[U]$~: $\sigma(u_{i,j}(z))=u_{i,j}(z^q)=\tilde\sigma(u_{i,j})(z)$. Ainsi, la \pre{K}algèbre $K[U]$ est un \pre{\sigma}anneau intègre dont le corps des \pre{\sigma}constantes est égal à ${\RQ k}$. 

\begin{definition}\label{def:PV}
Une \pre{\sigma}algèbre $\AA A$ sur $K$ est un anneau de \emph{Picard-Vessiot} du système d'équations fonctionnelles~\eqref{eq:systintro} s'il existe une matrice fondamentale de solutions de~\eqref{eq:systintro} telle que $\AA A=K[\det(U)^{-1},U]$ et $\AA A$ ne contient aucun \pre{\sigma}idéal non nul.
\end{definition}

On peut toujours construire un anneau de Picard-Vessiot d'un système d'équations fonctionnelles de la forme~\eqref{eq:systintro} comme le quotient de la \pre{\tilde\sigma}algèbre $K[\det(X)^{-1},X]$ sur $K$ par un \pre{\tilde\sigma}idéal qui est maximal parmi les \pre{\tilde\sigma}idéaux de $K[\det(X)^{-1},X]$. Puisque le corps des \pre{\sigma}constantes de $K$ est supposé algébriquement clos, il y a unicité de l'anneau de Picard-Vessiot de~\eqref{eq:systintro}, à un isomorphisme de \pre{\tilde\sigma}algèbre sur $K$ près, \emph{cf.}~\cite[Proposition 1.9]{vdPS=97}. 

Les \pre{\sigma}idéaux maximaux d'une \pre{\sigma}algèbre sont radicaux et donc les anneaux de Picard-Vessiot n'ont pas d'éléments nilpotents, [\emph{ibidem}, Lemma 1.7].

Nous montrons que, $K$ étant de caractéristique zéro et de corps des \pre{\sigma}constantes $\RQ k$ algébriquement clos, il existe un anneau de Picard-Vessiot de~\eqref{eq:systintro} contenant les composantes de vecteurs solutions donnés à l'avance.

\begin{prop}\label{prop:PV}
Soit $1\le r\le N$ et $\QQ u_1,\dots,\QQ u_r$ des vecteurs solutions de~\eqref{eq:systintro}, alors il existe un anneau de Picard-Vessiot de~\eqref{eq:systintro} sur $K$, contenant les composantes de $\QQ u_1,\dots,\QQ u_r$. Si $\QQ u_1,\dots,\QQ u_r$ sont linéairement indépendants sur $\RQ k$, il existe une matrice fondamentale de solutions de~\eqref{eq:systintro} dont les $r$ première colonnes sont les vecteurs $\QQ u_1,\dots,\QQ u_r$.
\end{prop}
\begin{proof}
Le corps $\AA k=K(\QQ u_1,\dots,\QQ u_r)$ engendré sur $K$ par les composantes des vecteurs $\QQ u_i$ est naturellement muni d'une extension de l'endomorphisme $\sigma$, puisque ces vecteurs satisfont $\sigma(\QQ u_i)=A^{-1}\QQ u_i$, $i=1,\dots,r$. D'après~\cite[Proposition 2.1.4]{NPQ=12} il existe un anneau de Picard-Vessiot $\AA A$ de \eqref{eq:systintro} sur $\AA k$ (unique à isomorphisme près). La \pre{\sigma}algèbre $\AA A$ étant un anneau de Picard-Vessiot sur $\AA k$, n'a pas d'élément nilpotent. Soit $U$ une matrice fondamentale de solutions de~\eqref{eq:systintro} engendrant $\AA A=\AA k[\det(U)^{-1},U]$ sur $\AA k$. Grâce à la proposition~\ref{wronskien}, les vecteurs $\QQ u_i$ sont combinaisons linéaires à coefficients dans $\RQ k$ des colonnes de la matrice $U$ et donc appartiennent à $R:=K[\det(U)^{-1},U]\subset\AA A$. Cette \pre{\sigma}algèbre $R$, de type fini sur $K$, n'a pas d'élément nilpotent (puisque $\AA A$ n'en a pas) et son corps des \pre{\sigma}constantes est $\RQ k$ (corps des \pre{\sigma}constantes de $K$ supposé algébriquement clos), d'après~\cite[Lemma 1.8]{vdPS=97}. La proposition~\ref{wronskien} implique encore que toute matrice fondamentale de solutions de~\eqref{eq:systintro} dont les coefficients sont dans $\AA A$, engendre $K[\det(U)^{-1},U]$ sur $K$. Enfin, le corps $K$ est parfait (étant de caractéristique zéro), de corps des \pre{\sigma}constantes $\RQ k$ algébriquement clos. Ainsi, les hypothèses de \cite[Corollary 1.24]{vdPS=97} sont satisfaites et cet énoncé entraîne que $R$ est un anneau de Picard-Vessiot de~\eqref{eq:systintro} sur $K$, contenant les composantes de $\QQ u_1,\dots,\QQ u_r$.

Lorsque les vecteurs $\QQ u_1,\dots,\QQ u_r$ sont linéairement indépendants sur $\RQ k$, nous pouvons les substituer dans une matrice fondamentale de solutions, le déterminant et les coefficients n'étant multipliés que par des éléments de $\RQ k$.
\end{proof}

Un anneau de Picard-Vessiot est unique à isomorphisme près, mais il n'est pas intègre en général. En revanche, d'après~\cite[Corollary 1.16 et Lemma 1.26]{vdPS=97} il est toujours un produit d'anneaux de Picard-Vessiot intègres, tous isomorphes à un anneau de Picard-Vessiot sur $K$ du système d'équations fonctionnelles
\begin{equation}\label{eq:systitere}
\QQ u=A\sigma(A)\dots\sigma^{\ell-1}(A)\sigma^\ell(\QQ u)
\end{equation}
pour un certain $\ell\in\RQ N^\times$. En combinant ce fait avec la proposition~\ref{prop:PV} on obtient:

\begin{prop}\label{prop:PVitere}
Dans les notations de la proposition~\ref{prop:PV} il existe un entier positif $\ell$ et un \pre{\sigma^\ell}corps contenant $K(\QQ u_1,\dots,\QQ u_r)$, qui est le corps des fractions d'un anneau de Picard-Vessiot du système d'équations fonctionnelles~\eqref{eq:systitere}, intègre.
\end{prop}
\begin{proof}
Il suffit d'observer que les vecteurs $\QQ u_1,\dots,\QQ u_r$, étant solutions de~\eqref{eq:systintro}, sont également solutions de~\eqref{eq:systitere} et on applique la proposition~\ref{prop:PV} à cette situation. Le \pre{\sigma^\ell}anneau obtenu est isomorphe à un anneau de Picard-Vessiot de~\eqref{eq:systitere} qui, pour un choix convenable de $\ell$, est intègre.
\end{proof}

Dans la suite de cette section on fixe une matrice fondamentale de solutions $U=U(z)$ de~\eqref{eq:systintro} et on note $R=K[\det(U)^{-1},U]$. On a
$$R\otimes_{K}K[\det(X)^{-1},X] = R\otimes_{\RQ k}\RQ k[\det(Y)^{-1},Y]$$
où $Y=U^{-1}X$. On note $\AA I_0$ le \pre{\sigma}idéal de $R\otimes_{\RQ k}\RQ k[\det(Y)^{-1},Y]$ engendré par $\AA P\subset K[X]$:
$$\AA I_0 = (P(z,U(z)Y), P\in\AA P) \subset R\otimes_{\RQ k}\RQ k[\det(Y)^{-1},Y]
\enspace,$$
puis $\AA I=\AA I_0\cap\RQ k[\det(Y)^{-1},Y]$. Le \emph{groupe de Galois} $G$ de~\eqref{eq:systintro} sur $K$ est le sous-groupe algébrique de $\RR {Gl}(\RQ k^N)$ défini sur $\RQ k$ par $\AA I$, {\it voir}~\cite[\S1.2]{vdPS=97}. L'image de $U$ par un élément du groupe de Galois est encore une matrice fondamentale de solutions de~\eqref{eq:systintro}. Donc, tout élément de $G$ est représenté par une matrice $y=(y_{i,j})_{1\leq i,j\leq N}$ à coefficients dans $\RQ k$, de déterminant non nul, satisfaisant $P(z,Xy)\in\AA P$ pour tout $P(z,X)\in\AA P$. Il coïncide avec le groupe des automorphismes de la \pre{K}algèbre $K[X]$ laissant invariant $\AA P$ (c'est-à-dire les automorphismes de $K[U]$), qui commutent à l'action de $\sigma$.

Notons $K(U)$ l'anneau total des fractions de l'anneau $R$. La \emph{correspondance de Galois} associe à tout sous-groupe algébrique $H$ du groupe de Galois $G$ l'extension de Picard-Vessiot $K(U)^H\subset K(U)$ sur le corps des éléments de $K(U)$ fixés par tous les éléments de $H$. Réciproquement, à tout \pre{\sigma}anneau intermédiaire $K\subset\AA K\subset K(U)$, dont les éléments non inversibles sont diviseurs de zéro, est associé le sous-groupe de $G$ qui fixe les éléments de $\AA K$. Le groupe $H$ est le groupe de Galois de l'extension de Picard-Vessiot $\AA K=K(U)^H\subset K(U)$. Lorsque le sous-groupe $H$ est normal alors $K\subset K(U)^H$ est aussi une extension de Picard-Vessiot de groupe de Galois $G/H$, \emph{voir}~\cite[Theorem 1.29]{vdPS=97}.

Le schéma $\RR{Spec}(R)\to\RR{Spec}(K)$ est un torseur sous l'action du groupe de Galois $G$, \emph{voir}~\cite[Theorem 1.13]{vdPS=97}. En particulier, $R^G=K$, $R$ n'a pas d'idéal propre, non trivial, invariant par $G$ et, $K$ étant un \pre{C^1}corps, $R=\RQ k[G]\otimes_{\RQ k} K$, {\it voir}~\cite[Corollary 1.15 \& 1.18]{vdPS=97}. Ceci conduit à une description du groupe de Galois comme le plus petit sous-groupe algébrique $G$ de $\RR{Gl}(\RQ k^N)$ défini sur $\RQ k$ tel qu'il existe $B\in\RR{Gl}(K^N)$ satisfaisant $B^{-1}A\sigma(B)\in G(K)$, \emph{voir}~\cite[Proposition 1.21]{vdPS=97}. En particulier, si $A$ appartient à un sous-groupe $H(K)$ de $\RR{Gl}(K^N)$ défini sur $\RQ k$, alors $G\subset H$.

Si $K(U)$ est un corps (\emph{i.e.} $R$ est intègre) il y a correspondance entre les sous-groupes algébriques connexes du groupe de Galois et les corps intermédiaires algébriquement clos dans $K(U)$, \emph{voir}~\cite[Theorem 1.29 et page 22]{vdPS=97}. De plus, la dimension (sur $\RQ k$) du groupe algébrique associé à un \pre{\sigma}corps intermédiaire $K\subset\AA K\subset K(U)$ est égal au degré de transcendance de l'extension de Picard-Vessiot $\AA K\subset K(U)$.

L'action du groupe de Galois sur l'espace des matrices fondamentales du système d'équations fonctionnelles considéré est donné par la multiplication à droite par la matrice $y$ représentant l'élément du groupe de Galois~: $U\mapsto Uy$. Cela induit un endomorphisme de l'espace des solutions $\SS S=\oplus_{i=1}^N\RQ k\QQ u_i$
$$\begin{matrix}
\SS S&\longrightarrow &\SS S\\
U\RQ c=\sum_{i=1}^Nc_i\QQ u_i&\longmapsto &Uy\QQ c=\sum_{i=1}^N\left(\sum_{j=1}^Ny_{i,j}c_j\right)\QQ u_{i}
\end{matrix}\enspace,$$
où $\QQ c$ désigne le vecteur colonne transposé de $(c_1,\dots,c_N)\in\RQ k^N$. Comme $\sigma(Uy\QQ c)=\sigma(U)y\QQ c$, cet endomorphisme commute à $\sigma$ et sa matrice dans la base $(\QQ u_i)_{1\leq i\leq N}$ est $y$. 

\medskip

Notons $G^\circ$ la composante connexe contenant l'origine de $G$, c'est un sous-groupe distingué de $G$. D'après~\cite[Corollary 1.31]{vdPS=97}, le sous-anneau $R^{G^\circ}$ des éléments de $R$ fixés par $G^\circ$ est un \pre{K}espace vectoriel de dimension finie, égale à l'indice $[G:G^\circ]$ de $G^\circ$ dans $G$. Lorsque $R$ est intègre on a le résultat suivant~:
\begin{lem}\label{lem:regulariteetconnexite}
Supposons l'anneau de Picard-Vessiot $R$ intègre, alors le corps $K(U)^{G^\circ}$ est la clôture algébrique de $K$ dans $K(U)$, c'est une extension finie de $K$ de degré $[G:G^\circ]$. En particulier, $K$ est algébriquement clos dans $K(U)$ si et seulement si le groupe de Galois de~\eqref{eq:systintro} est connexe.
\end{lem}
\begin{proof}
Comme on suppose $R$ intègre et que $G^\circ$ est connexe et distingué, $K(U)^{G^\circ}$ est un sous-corps algébriquement clos dans $K(U)$ et $K\subset K(U)^{G^\circ}$ est une extension de Picard-Vessiot de groupe de Galois $G/G^\circ$, donc algébrique finie de degré $[G:G^\circ]$.

Ce qui montre que $K(U)^{G^\circ}$ est la clôture algébrique de $K$ dans $K(U)$ (qui est nécessairement un \pre{\sigma}corps). 

Ainsi, $K(U)^{G^\circ}=K$ si et seulement si $G^\circ=G$.
\end{proof}

\section{Régularité et spécialisation}\label{sec:regspe}
Nous supposons dorénavant $\RQ k=\OO{\RQ Q}$ et nous considérons le système d'équations fonctionnelles~\eqref{eq:systintro}, où $A(z)$ est une matrice $N\times N$ à coefficients dans $\OO{\RQ Q}(z)$. Soit $\rho$ un réel positif, $1\le r\le N$ un entier et $\QQ u_1(z),\dots,\QQ u_r(z)$ des vecteurs solutions de~\eqref{eq:systintro} dont les composantes $u_{i,j}(z)$, $i=1,\dots,N$, $j=1,\dots,r$, appartiennent à $\OO{\RQ Q}[[z]]$ et convergent dans un disque de rayon $\rho$ centré en l'origine de $\RQ C$. Rappelons que les séries $u_{i,j}(z)$ définissent des fonctions méromorphes dans le disque unité sans bord de $\RQ C$, \emph{voir} lemme~\ref{lem:convergence} à la section~\ref{sec:convergence}.

Posons $V=V(z):=(u_{i,j}(z))_{\stackrel{1\le i\le N}{\scriptscriptstyle 1\le j\le r}}$ et rappelons qu'on note $T(z)=T_A(z)$ le polynôme de $\OO{\RQ Q}[z]$ satisfaisant $T(0)=1$ et dont les zéros sont les pôles distincts de $0$, des matrices $A(z)$ et $A(z)^{-1}$, comptés avec multiplicité. Les \emph{singularités} du système~\eqref{eq:systintro} sont les points $\xi\in\RQ C$ satisfaisant $T(\xi^{q^\ell})=0$ pour un $\ell\in\RQ N$. On vérifie qu'un pôle d'une des fonctions $u_{i,j}(z)$ dans le disque unité sans bord de $\RQ C$, est une singularité du système~\eqref{eq:systintro}.

Appliquons le théorème~\ref{thm:nishioka} à des sommes directes du système~\eqref{eq:systintro}.

\begin{thm}\label{thm:indalgbis}
Soit $A(z)$ une matrice carrée $N\times N$ à coefficients dans $\OO{\RQ Q}(z)$. Soit $\QQ u_1(z),\dots,\QQ u_r(z)$ des vecteurs solutions de~\eqref{eq:systintro} dont les composantes $u_{i,j}(z)$, $i=1,\dots,N$, $j=1,\dots,r$, appartiennent à $\OO{\RQ Q}[[z]]$. 

Soit $\alpha\in\OO{\RQ Q}\subset\RQ C$, $0<|\alpha|<\rho$, qui n'est pas une singularité de~\eqref{eq:systintro}, alors
$$\RR{degtr}_{\OO{\RQ Q}}\OO{\RQ Q}(V(\alpha)) \geq \RR{degtr}_{\OO{\RQ Q}(z)}\OO{\RQ Q}(z,V(z))\enspace.$$.
\end{thm}
\begin{proof}
On applique le théorème~\ref{thm:nishioka} au système $rN\times rN$ de matrice $B(z)$ composée de $r$ blocs diagonaux $A(z)$, dont le vecteur colonne de composantes les $u_{i,j}(z)$ est solution. Par hypothèses les séries $u_{i,j}(z)$ convergent dans le disque de rayon $\rho$ centré en l'origine de $\RQ C$ et leurs coefficients appartiennent à l'extension finie de $\RQ Q$ engendrée par les coefficients des coefficients de $A(z)$. De plus, les pôles des coefficients de la matrice $B^{-1}(z)$ sont les mêmes que ceux de $A^{-1}(z)$ et les hypothèses du théorème~\ref{thm:nishioka} se ramènent donc à celles du théorème~\ref{thm:indalgbis}.
\end{proof}

Lorsqu'un corps est algébriquement clos dans un plus grand corps, on dit que l'extension est \emph{régulière}. Le lemme~\ref{lem:extreg} suivant peut donc se reformuler~: $K\subset K(V)$ et $\OO{\RQ Q}(z)\subset\OO{\RQ Q}(z,V)$ sont des extensions régulières.

\begin{lem}\label{lem:extreg}
Tout élément de $K(V)$ ({\it resp.} $\OO{\RQ Q}(z,V)$) algébrique sur $K$ ({\it resp.} $\OO{\RQ Q}(z)$) appartient à $K$ ({\it resp.} $\OO{\RQ Q}(z)$).
\end{lem}
\begin{proof}
Grâce à la proposition~\ref{prop:PVitere}, le corps $K(V)$ est contenu dans une extension de Picard-Vessiot de $K$ pour le système d'équations fonctionnelles~\eqref{eq:systitere}. Soit $L$ la clôture algébrique de $K$ dans $K(V)$, elle est contenue dans la clôture algébrique de $K$ dans l'extension de Picard-Vessiot exhibée qui, d'après le lemme~\ref{lem:regulariteetconnexite}, est une extension finie de $K$. L'extension $K\subset L$ est donc elle aussi finie. Soit $f\in L$, alors $\sigma^\ell(f)$ est encore un élément de $K(V)$ algébrique sur $K$, pour tout $\ell\in\RQ N$. Comme $L$ est un \pre{K}espace vectoriel de dimension finie, il existe donc, pour un certain $m\in\RQ N^\times$, une relation de la forme
\begin{equation}\label{eq:fonctionnellef}
\sigma^m(f)=\sum_{\ell=0}^{m-1}b_\ell\sigma^\ell(f)
\end{equation}
avec $b_\ell\in K$ pour $\ell=0,\dots,m-1$. Notons $\SS R$ l'ensemble (fini) des points de ramification de $f$ et des $b_\ell$, on note que ces derniers $b_\ell$ sont méromorphes dans $\RQ C^\times$ et n'ont de ramification qu'en $0$ et $\infty$. Soit $\theta$ l'argument minimal dans $]0,2\pi]$ d'un élément de $\SS R\cap\{z\in\RQ C;|z|=1\}$ (en posant $\theta=2\pi$ si ce dernier ensemble est vide). Pour tout réel $0<\varepsilon<1$ notons encore
$$E(\varepsilon) := \left\{z\in\RQ C;\theta-\varepsilon<\RR{Arg}(z)<\theta,1-\varepsilon<|z|<1+\varepsilon\right\}
\enspace.$$
Lorsque $\varepsilon$ est assez petit l'ensemble $\bigcup_{h=0}^{\infty} E(\varepsilon)^{q^{-h}}$ ne rencontre pas $\SS R$, grâce au choix de $\theta$. La fonction algébrique $f$ étant méromorphe en dehors de $E$, l'est en particulier dans $\bigcup_{h=0}^{\infty} E(\varepsilon)^{q^{-h}}$.

Maintenant, si $f$ est une fonction méromorphe dans un ensemble de la forme $\bigcup_{h=0}^{\infty} E_0^{q^{-h}}$ avec $E_0\subset\RQ C^\times$, alors pour chaque $\ell=0,\dots,m-1$ la fonction $\sigma^\ell(f)$ est méro\-mor\-phe dans $\bigcup_{h=\ell}^{\infty} E_0^{q^{-h}}$. Et donc, par la relation~\eqref{eq:fonctionnellef}, la fonction $\sigma^m(f)$ est méromorphe dans $\bigcup_{h=m-1}^{\infty} E_0^{q^{-h}}$, ce qui implique finalement la méromorphie de $f$ dans $\bigcup_{h=-1}^{\infty} E_0^{q^{-h}}$. Appliquant ce raisonnement à $E_0=E(\varepsilon)$, puis itérant avec $E_0=E(\varepsilon)^q, E(\varepsilon)^{q^2}, \dots$ successivement, on montre que la fonction $f$ (et $\sigma^\ell(f)$ pour tout $\ell\in\RQ N$) est méromorphe dans $\bigcup_{h\in\RQ Z} E(\varepsilon)^{q^{-h}}=\RQ C^\times$.

Mais, $f\in L\subset K(V)$ appartient à un corps $\OO{\RQ Q}(z^{q^{-\ell}})(V)\subset\OO{\RQ Q}(z^{q^{-\ell}})[[z]]$ pour un certain $\ell\in\RQ N$, ce qui entraîne que $\sigma^\ell(f)$ n'est pas ramifiée en $0$. Avec ce qui précède, elle est donc méromorphe dans $\RQ C$ et la formule de Riemann-Hurwitz impose qu'elle soit rationnelle. Ainsi $\sigma^\ell(f)\in\OO{\RQ Q}(z)$, d'où suit $f\in K$.

Si on suppose $f\in L\cap\OO{\RQ Q}(z,V)\subset\OO{\RQ Q}[[z]]$ alors le raisonnement ci-dessus montre que la fonction $f$ elle-même est méromorphe dans $\RQ C$ et donc appartient à $\OO{\RQ Q}(z)$.
\end{proof}

\begin{undef}{Remarque}
La démonstration ci-dessus est à rapprocher de celle du lem\-me~5 de~\cite{Bk=94}. On y ajoute comme ingrédient supplémentaire, en s'appuyant sur la théorie de Galois, que la clôture algébrique de $K$ dans $K(V)$ est une extension finie de $K$.
\end{undef}

Rappelons la matrice de variables $X:=(X_{i,j})_{\stackrel{1\le i\le N}{\scriptscriptstyle 1\le j\le r}}$ et l'automorphisme $\tilde\sigma$ introduit à la fin de la section~\ref{sec:corpsendo}. On note aussi $\AA P$ l'idéal premier de $K[X]$ des relations algébriques entre les coefficients de la matrice $V$. Adjoignons une nouvelle variable $X_0$ à $X$ et notons $\tilde X=(X,X_0)$. Soit $\tilde{\AA P}$ l'idéal de $K[\tilde X]$ homogénéisé en $X$ de l'idéal $\AA P$, c'est un idéal premier de même rang que $\AA P\subset K[X]$.

\begin{prop}\label{prop:indalg}
Il existe un réel $0<\rho'<\rho$ tel que tout $\alpha\in\OO{\RQ Q}\subset\RQ C$, $0<|\alpha|\leq\rho'$, ait la propriété suivante~: pour tout polynôme $P\in\OO{\RQ Q}[X]$ tel que $P(V(\alpha))=0$, il existe $Q\in\OO{\RQ Q}[z,X]$ de même degré en $X$ que $P$, satisfaisant $Q(z,V(z))=0$ et $P(X)=Q(\alpha,X)$.
\end{prop}
\begin{proof}
Soit $\AA P_\alpha$ l'idéal premier de $\OO{\RQ Q}[X]$ des relations de dépendance algébrique (à coefficients algébriques) entre les coefficients de $V(\alpha)$ et $\tilde{\AA P}_\alpha$ son homogénéisé dans $\OO{\RQ Q}[\tilde X]$. Nous notons $\AA s_\alpha:\OO{\RQ Q}[z]\to\OO{\RQ Q}$ la spécialisation de $z$ en $\alpha$, nous allons montrer $\AA s_\alpha(\tilde{\AA P}\cap \OO{\RQ Q}[z,\tilde X])=\tilde{\AA P}_\alpha$ pour l'application étendue $\AA s_\alpha:\OO{\RQ Q}[z,\tilde X]\to\OO{\RQ Q}[\tilde X]$. Notons qu'on a clairement $\AA s_\alpha(\tilde{\AA P}\cap \OO{\RQ Q}[z,\tilde X])\subset\tilde{\AA P}_\alpha$ et que les idéaux premiers $\AA P_\alpha$ et $\tilde{\AA P}_\alpha$ ont même rang.

Comme $\RR{Frac}(\OO{\RQ Q}[z,X]/\AA P\cap \OO{\RQ Q}[z,X]) \simeq \OO{\RQ Q}(z,V)$ est une extension régulière de $\OO{\RQ Q}(z)$ d'après le lemme~\ref{lem:extreg}, l'idéal $\AA P\cap \OO{\RQ Q}(z)[X]$ est absolument premier, voir~\cite[Theorem 39, p.230]{ZS=60}. L'idéal $\tilde{\AA P}\cap\OO{\RQ Q}(z)[\tilde X]$ est l'homogénéisé en $X$ de $\AA P\cap \OO{\RQ Q}(z)[X]$, il est donc également absolument premier de même rang. Et, d'après~\cite[\S3, Satz 16]{Kr=48}, son image par $\AA s_\alpha$ est un idéal premier de $\OO{\RQ Q}[\tilde X]$ de même rang que $\tilde{\AA P}\cap\OO{\RQ Q}(z)[\tilde X]$, pour tout $\alpha$ en dehors d'un ensemble fini. On choisit $0<\rho'<\rho$ strictement plus petit que le plus petit des modules des éléments non nuls de cet ensemble auquel on adjoint les zéros de $T_A(z)$. On suppose dorénavant $0<|\alpha|\leq\rho'$, de sorte que $\alpha$ n'est pas une singularité de~\eqref{eq:systintro} et $\AA s_\alpha(\tilde{\AA P}\cap\OO{\RQ Q}[z,\tilde X])$ est un idéal premier de même rang que $\tilde{\AA P}\cap\OO{\RQ Q}(z)[\tilde X]$ et donc que ${\AA P}\cap\OO{\RQ Q}(z)[X]$.

L'anneau $\OO{\RQ Q}(z)[V]$ est de type fini sur $\OO{\RQ Q}(z)$ et intègre, aussi sa dimension de Krull est-elle égale au degré de transcendance sur $\OO{\RQ Q}(z)$ du corps $\OO{\RQ Q}(z,V)$, voir~\cite[\S8.2.1, Theorem A, pp.221]{Ei=95}. De plus, comme l'idéal $\AA s_\alpha(\tilde{\AA P}\cap \OO{\RQ Q}[z,\tilde X])$ a même rang que $\AA P\cap\OO{\RQ Q}(z)[X]$, on a
$$\begin{aligned}
\RR{degtr}_{\OO{\RQ Q}(z)}\OO{\RQ Q}(z,V) &= \dim\left(\OO{\RQ Q}(z)[V]\right) = \dim\left(\OO{\RQ Q}(z)[X]/\AA P\cap \OO{\RQ Q}(z)[X]\right)\\
&= \dim(\OO{\RQ Q}(z)[X])-\RR{rg}(\AA P\cap\OO{\RQ Q}(z)[X])\\
&= \dim(\OO{\RQ Q}[\tilde X])-\RR{rg}(\AA s_\alpha(\tilde{\AA P}\cap \OO{\RQ Q}[z,\tilde X]))-1
\enspace.\end{aligned}$$

De même, l'anneau $\OO{\RQ Q}[X]/\AA P_\alpha$ étant de type fini sur $\OO{\RQ Q}$ et intègre, sa dimension de Krull est égale au degré de transcendance sur $\OO{\RQ Q}$ du corps $\OO{\RQ Q}(V(\alpha))$, voir~{\it ibidem}, on a
$$\RR{degtr}_{\OO{\RQ Q}}\OO{\RQ Q}(V(\alpha)) = \dim(\OO{\RQ Q}[X])-\RR{rg}(\AA P_{\alpha}) = \dim(\OO{\RQ Q}[\tilde X])-\RR{rg}(\tilde{\AA P}_\alpha)-1
\enspace.$$
Le théorème~\ref{thm:indalgbis} entraîne donc
$$\RR{rg}(\AA s_\alpha(\tilde{\AA P}\cap \OO{\RQ Q}(z)[\tilde X])) \ge \RR{rg}(\tilde{\AA P}_\alpha)$$
et, comme $\AA s_\alpha(\tilde{\AA P}\cap \OO{\RQ Q}[z,\tilde X])\subset \tilde{\AA P}_\alpha$, on en déduit que les idéaux premiers $\AA s_\alpha(\tilde{\AA P}\cap \OO{\RQ Q}[z,\tilde X])$ et $\tilde{\AA P}_\alpha$ sont égaux, comme voulu.

Ainsi, le polynôme $\tilde P\in\tilde{\AA P}_\alpha\subset\OO{\RQ Q}[\tilde X]$, homogénéisé de $P\in\OO{\RQ Q}[X]$, s'écrit $\AA s_\alpha(\tilde Q)$ pour un polynôme $\tilde Q\in\tilde{\AA P}\cap\OO{\RQ Q}[z,\tilde X]$, de même degré en $\tilde X$ que $\tilde P$. Le polynôme $Q\in\AA P\cap\OO{\RQ Q}[z,X]$ obtenu de $\tilde Q$ par spécialisation de $X_0$ en $1$ satisfait la conclusion de la proposition.
\end{proof}

Le corollaire suivant établit le théorème~\ref{thm:indalg}.

\begin{cor}\label{cor:indalgi}
Soit $\alpha\in\OO{\RQ Q}\subset\RQ C$, $0<|\alpha|<1$, n'appartenant pas à l'ensemble singulier de~\eqref{eq:systintro}, alors pour tout $P\in\OO{\RQ Q}[X]$ tel que $P(V(\alpha))=0$, il existe $Q\in\OO{\RQ Q}[z,X]$ de même degré en $X$ que $P$, tel que $Q(z,V(z))=0$ et $P(X)=Q(\alpha,X)$.
\end{cor}
Lorsque l'idéal $\AA P$ est homogène en $X$ et $P$ est un élément homogène de $\AA P_\alpha$, on peut choisir le polynôme $Q$ également homogène en $X$, de même degré que $\AA P$.
\begin{proof}
Il existe un entier $\ell\in\RQ N$ tel que $0<|\alpha^{q^\ell}|\leq\rho'$ dans les notations de la proposition~\ref{prop:indalg}. Comme $\alpha$ n'est pas singularité de~\eqref{eq:systintro} la matrice $A^{(\ell)}(\alpha)$ est définie et inversible. Maintenant, si $P\in\OO{\RQ Q}[X]$ satisfait $P(V(\alpha))=0$ alors $P_0(X):=P(A^{(\ell)}(\alpha)X)\in\OO{\RQ Q}[X]$ satisfait $P_0(V(\alpha^{q^\ell}))=0$, car $V(\alpha)=A^{(\ell)}(\alpha)V(\alpha^{q^\ell})$ d'après~\eqref{eq:systintro} itérée (\emph{i.e.}~\eqref{eq:systitere}). De plus, $P_0$ a même degré en $X$ que $P$. Comme $|\alpha^{q^\ell}|\le\rho'$, la proposition~\ref{prop:indalg} entraîne l'existence de $Q_0\in\AA P\cap \OO{\RQ Q}[z,X]$ de même degré en $X$ que $P_0$, tel que $P_0(X)=Q_0(\alpha^{q^\ell},X)$.

Appliquons $\tilde\sigma^\ell$ à $P_0$, on obtient d'une part
$$\tilde\sigma^\ell(P_0(X)) = \tilde\sigma^\ell(Q_0(\alpha^{q^\ell},X)) = Q_0(\alpha^{q^\ell},A^{(-\ell)}(z)X)
$$
et d'autre part, par la définition de $P_0$,
$$\tilde\sigma^\ell(P_0(X)) = \tilde\sigma^{\ell}(P(A^{(\ell)}(\alpha)X)) = P(A^{(\ell)}(\alpha)A^{(-\ell)}(z)X)
\enspace.$$
On en déduit
\begin{equation}\label{eq:dspreuvecorindalg}
P(A^{(\ell)}(\alpha)A^{(-\ell)}(z)X) = Q_0(\alpha^{q^\ell},A^{(-\ell)}(z)X)
\enspace.\end{equation}

En multipliant $Q_0(\alpha^{q^\ell},A^{(-\ell)}(z)X)\in \OO{\RQ Q}(z)[X]$ par une puissance convenable de $\Delta(z):=T_A(z)T_A(z^q)\dots T_A(z^{q^{\ell-1}})$ on obtient un polynôme en $z$ et $X$ et, vu la contrainte imposée à $\alpha$, on a $\Delta(\alpha)\not=0$. Donc, la spécialisation de $z$ à $\alpha$ dans $Q_0(\alpha^{q^\ell},A^{(-\ell)}(z)X)$ est bien définie et, avec~\eqref{eq:dspreuvecorindalg}, donne l'égalité
$$P(X) = Q_0(\alpha^{q^\ell},A^{(-\ell)}(\alpha)X)
\enspace,$$
car $A^{(\ell)}(\alpha)A^{(-\ell)}(\alpha)=\RR{Id}$. Mais $Q_0\in\AA P$ et $\AA P$ est stable par l'action de $\tilde\sigma^\ell$, aussi $\tilde\sigma^\ell(Q_0) = Q_0(z^{q^\ell},A^{(-\ell)}(z)X)\in\AA P\cap \OO{\RQ Q}(z)[X]$. En posant
$$Q(z,X) = \left(\frac{\Delta(z)}{\Delta(\alpha)}\right)^{\RR d^\circ_XQ_0}\cdot Q_0(z^{q^\ell},A^{(-\ell)}(z)X)$$
on a $Q\in\OO{\RQ Q}[z,X]$, $Q(z,V(z))=0$, $P(X)=Q(\alpha,X)$ et le degré en $X$ de $Q$ est le même que celui des polynômes $Q_0$, $P_0$ et $P$, comme requis.
\end{proof}

\begin{undef}{Remarque}
La structure des démonstrations des corollaires~\ref{prop:indalg} et~\ref{cor:indalgi} ci-dessus s'apparente à celle de~\cite[Theorem 1.6.1]{An=12}. La notion d'idéal absolument premier ici correspond à celle de schéma géométriquement intègre là. Et c'est l'action du groupe de Galois différentiel qui dans~\cite{An=12} assure, au travers de la notion de fibré algébrique, l'uniformité de l'intégrité des fibres, ici remplacée dans la preuve du corollaire~\ref{cor:indalgi} par l'action de l'endomorphisme $\tilde\sigma$. 
\end{undef}

\begin{cor}\label{cor:indalgii}
Supposons que le degré de transcendance de $\OO{\RQ Q}(V(z))$  sur $\OO{\RQ Q}$ soit égal à son degré de transcendance sur $\OO{\RQ Q}(z)$. Autrement-dit, on suppose $z$ transcendant sur $\OO{\RQ Q}(V(z))$.

Soit $\alpha\in\OO{\RQ Q}\subset\RQ C$, $0<|\alpha|<1$, n'appartenant pas à l'ensemble singulier de~\eqref{eq:systintro}, alors les relations de dépendance algébrique sur $\OO{\RQ Q}$ entre les coefficients de $V(\alpha)$ sont les relations de dépendance algébrique sur $\OO{\RQ Q}$ entre les coefficients de $V(z)$.
\end{cor}
\begin{proof}
Il résulte du corollaire~\ref{cor:indalgi} et de l'hypothèse de l'énoncé qu'on a $\AA P_\alpha=\AA P\cap\OO{\RQ Q}[X]$, car $\AA P$ est engendré par $\AA P\cap\OO{\RQ Q}[X]$.
\end{proof}

\section{Indépendance linéaire}\label{sec:independancelin}
Dans cette section on reprend les notations des paragraphes précédents avec $\RQ k=\OO{\RQ Q}$ et $K$ sa clôture inversive.

Notre but est d'obtenir des énoncés d'indépendance linéaire sur $\OO{\RQ Q}$ des valeurs des composantes d'un vecteur solution de~\eqref{eq:systintro}. Notamment, de démontrer le théorème~\ref{thm:indlinautsing}, avec en particulier l'enjeu de ne pas exclure l'ensemble singulier du système d'équations fonctionnelles.

Commençons par un lemme reliant indépendance linéaire des fonctions et orbite galoisienne.

\begin{lem}\label{lem:indlin}
Soit $\QQ f(z)\in \RQ k[[z]]$ un vecteur solution de~\eqref{eq:systintro} et $\SS S$ l'espace des solutions engendré par les colonnes $\QQ u_1(z),\dots,\QQ u_N(z)$ d'une matrice fondamentale de solutions de~\eqref{eq:systintro}, contenant $\QQ f(z)$. 

La dimension de la sous-représentation du groupe de Galois sur le \pre{\RQ k}espace vectoriel $\SS S'\subset\SS S$ engendré par l'orbite de $\QQ f(z)$ est égale à la dimension du \pre{K}espace vectoriel engendré par les composantes de $\QQ f(z)$. 
\end{lem}
\begin{proof}
L'existence d'une extension de Picard-Vessiot $R$ engendrée par une matrice fondamentale de solutions $U(z)=(\QQ u_1(z),\dots,\QQ u_N(z))$ de \eqref{eq:systintro} telle que $\QQ u_1(z)=\QQ f(z)$, est assurée par la proposition~\ref{prop:PV}. 

Soit $r$ la dimension de l'orbite de $\QQ f(z)$ sous l'action du groupe de Galois et $\QQ f(z)=\QQ u'_1(z),\QQ u'_2(z),\dots,\QQ u'_r(z)\in R^N$ une base de $\SS S'$ composée d'images de $\QQ f(z)$ par des éléments du groupe de Galois. Notons $\boldsymbol{\lambda}_1,\dots,\boldsymbol{\lambda}_{r}\in\RQ k^N$ les vecteurs des coordonnées de $\QQ u'_1(z),\dots,\QQ u'_r(z)$ respectivement, dans la base $\QQ u_1(z),\dots,\QQ u_N(z)$ de $\SS S$ et $\boldsymbol{\mu}_1,\dots,\boldsymbol{\mu}_{N-r}\in\RQ k^N$ un ensemble complet de vecteurs orthogonaux aux $\boldsymbol{\lambda}_i$. Alors, un vecteur de $K(U)^N$ orthogonal à $\QQ u'_1(z),\dots,\QQ u'_r(z)$ est combinaison linéaire des vecteurs ${}^tU(z)^{-1}\boldsymbol{\mu}_i$, qui sont linéairement indépendants sur $K(U)$. Ainsi la dimension de l'espace vectoriel engendré sur $K(U)$ par $\QQ u'_1(z),\dots,\QQ u'_r(z)$ est égal à la dimension de $\SS S'$ sur $\RQ k$ et la matrice $N\times r$ des composantes des vecteurs $\QQ u'_1(z),\dots,\QQ u'_{r}(z)$, notée $\SS M$ dans la suite, est de rang $r$. Quitte à réordonner les coordonnées on peut supposer que le mineur $\Delta$ des $r$ premières lignes de $\SS M$ est non nul. Pour $1\le i\le r$ et $r+1\le j\le N$ notons $\Delta_{i,j}$ le mineur déduit de $\Delta$ par substitution de la \pre{j}ième ligne de $\SS M$ en lieu et place de sa \pre{i}ième ligne. 

Un élément $y$ du groupe de Galois agit sur $\SS S'$ et transforme la matrice $\SS M$ par multiplication à droite par une matrice $r\times r$ à coefficients dans $\RQ k$, notée $y'$. L'action de $y$ multiplie ainsi chaque $\Delta_{i_1,\dots,i_r}$ par $\det(y')$. Il suit que le quotient de deux tels mineurs est invariant par le groupe de Galois et appartient donc à $K$. Une base du \pre{K}espace vectoriel des vecteurs de $K^N$ orthogonaux à $\QQ u'_1(z),\dots,\QQ u'_r(z)$ est donnée par les lignes de la matrice
$$\left(\begin{array}{ccccccc}
\scriptstyle\Delta_{1,r+1}/\Delta &\dots &\scriptstyle\Delta_{r,r+1}/\Delta &\scriptstyle1 &\scriptstyle0 &\dots &\scriptstyle0\\
\vdots & &\vdots &\scriptstyle0 &\ddots &\ddots &\vdots\\[-1mm]
\vdots & &\vdots &\vdots &\ddots &\ddots &\scriptstyle0\\
\scriptstyle\Delta_{1,N}/\Delta &\dots &\scriptstyle\Delta_{r,N}/\Delta &\scriptstyle0 &\dots &\scriptstyle0 &\scriptstyle1\\
\end{array}\right)
\enspace.$$
Maintenant, ce \pre{K}espace vectoriel coïncide avec le \pre{K}espace vectoriel des relations entre les composantes de $\QQ f(z)$, qui est donc de dimension $N-r$. Il suit que $r$ est égal à la dimension du \pre{K}espace vectoriel engendré par les composantes de $\QQ f(z)$.
\end{proof}

On peut alors énoncer~:

\begin{cor}\label{cor:indlin}
On suppose qu'il existe une matrice fondamentale de solutions $U(z)$ de~\eqref{eq:systintro} dont tous les coefficients sont analytiques dans le disque unité sans bord de $\RQ C$. Soit $\QQ f(z)$ un vecteur solution de~\eqref{eq:systintro} dont les composantes $f_1(z),\dots,f_N(z)$ sont dans $\OO{\RQ Q}(z,U(z))$ et linéairement indépendantes sur $\OO{\RQ Q}(z)$.

Alors, pour $\alpha\in\OO{\RQ Q}\subset\RQ C$, $0<|\alpha|<1$, qui n'est pas une singularité de~\eqref{eq:systintro}, les nombres $f_1(\alpha),\dots, f_N(\alpha)$ sont linéairement indépendants sur $\OO{\RQ Q}$.
\end{cor}
\begin{proof}
D'après la proposition~\ref{wronskien} il existe $\QQ\mu\in\OO{\RQ Q}^N$ tel que $\QQ f(z)=U(z)\QQ\mu$. En particulier, les fonctions $f_1(z),\dots,f_N(z)$ sont analytiques en l'origine et on en déduit que l'hypothèse d'indépendance linéaire sur $\OO{\RQ Q}(z)$ est équivalente à leur indépendance linéaire sur la clôture inversive $K$ de $\OO{\RQ Q}(z)$.

Soit $\langle\QQ\lambda,\QQ f(\alpha)\rangle=0$ une relation avec $\QQ\lambda\in\OO{\RQ Q}^N$ non nul. D'après le corollaire~\ref{cor:indalgi} elle provient d'une relation $\langle\QQ\lambda,\QQ f(z)\rangle=(z-\alpha)R(z,U(z))$ avec $R\in\OO{\RQ Q}[z,X]$, qui entraîne, pour tout $y\in G\subset\RR{Gl}(\RQ k^N)$,
$$\langle\QQ\lambda,U(z)y\QQ\mu\rangle=(z-\alpha)R(z,U(z)y)
\enspace,$$
puis $\langle\QQ\lambda,U(\alpha)y\QQ\mu\rangle=0$. D'après le lemme~\ref{lem:indlin} l'indépendance linéaire sur $K$ des composantes de $\QQ f(z)$ implique que l'orbite de $\QQ f(z)$ sous l'action du groupe de Galois $G$ est de dimension maximale $N$ et donc que les vecteurs $y\QQ\mu$ engendent $\OO{\RQ Q}^N$ lorsque $y$ parcourt $G$. Comme le vecteur $\QQ\lambda$ n'est pas nul, ceci entraîne $\det(U(\alpha))=0$, car sinon les vecteurs $U(\alpha)y\QQ\mu$, $y\in G$, engendreraient $\OO{\RQ Q}^N$. Soit $\ell\in\RQ N$ le plus grand entier  tel que $\det(U(\alpha^{q^\ell}))=0$, il suit de~\eqref{eq:systintro} et $\det(U(z))=\det(A(z))\det(U(z^q))$ que $\det(A(\alpha^{q^\ell}))=0$. Ainsi $\alpha^{q^\ell}$ est un pôle de $\det(A(z)^{-1})$ et donc de $A(z)^{-1}$, contrairement à l'hypothèse.
\end{proof}

Voici deux exemples montrant que l'hypothèse d'existence d'une matrice fondamentale de solutions analytique dans le corollaire~\ref{cor:indlin} ne peut être supprimée et que celle d'indépendance linéaire de toutes les composantes du vecteur solution ne peut être restreinte à une sous-famille de composantes.
\begin{exes}\label{exe:irreginfini}
1) Soit $q\ge 3$ un entier et $\alpha$ un nombre complexe satisfaisant $0<|\alpha|<1$, considérons le système d'équation fonctionnelles pour la transformation $z\mapsto z^q$
$$\begin{pmatrix}-\Delta(z)-1 &-\Delta(z)\\ 1 &0\end{pmatrix}
\enspace,$$
où $\Delta(z)=-\frac{1-z/\alpha^q}{(1-z^q/\alpha)(1+\dots+(z/\alpha)^{q-1})}$ est le déterminant du système. On observe que $\Delta(\alpha)=\frac{1}{q\alpha^{q-1}}$ et $\alpha$ n'est pas une singularité du système d'équations fonctionnelles. 

Une matrice fondamentale de solutions de ce système s'écrit
\begin{equation}\label{eq:matfondsiii}
\begin{pmatrix}f(z) &u(z)\\ f(z^q) &u(z^q)\end{pmatrix}
\enspace,\end{equation}
où les fonctions $f(z)$ et $u(z)$ sont données par les développements
$$\begin{cases}
f(z) = \frac{1}{2} - \big(\frac{1}{\alpha}+\frac{1}{\alpha^q}\big)g(z) + \frac{1}{\alpha^{q+1}}g(z^2)\enspace,\quad g(z) = \sum_{i\in\RQ N}(-1)^iz^{q^{i}}\\[2mm]
u(z) = \frac{1}{2} - \big(\frac{1}{\alpha}+\frac{1}{\alpha^q}\big)v(z) + \frac{1}{\alpha^{q+1}}v(z^2)\enspace,\quad v(z) = \sum_{i\in\RQ N}(-1)^{i}z^{q^{-i-1}}
\end{cases}
\enspace.$$
On note que $u(z)$ et $v(z)$ ne sont analytiques en aucun point. Les fonctions $g(z)$ et $v(z)$ satisfont $g(z)+g(z^q)=v(z)+v(z^q)=z$, d'où
\begin{multline}\label{eq:dsexiii}
f(z)+f(z^q) = u(z)+u(z^q) =\\ 1 - \Big(\frac{1}{\alpha}+\frac{1}{\alpha^q}\Big)z + \frac{1}{\alpha^{q+1}}z^2 = \left(1-\frac{z}{\alpha}\right)\left(1-\frac{z}{\alpha^q}\right)
\end{multline}
et aussi que $g(z)$ et $g(z^q)$ sont algébriquement indépendantes sur $K$. On observera que si $h(z)=\sum_{i\in\RQ N}z^{q^{2i}}$ alors $g(z)=h(z)-h(z^q)$ et, comme $q\ge3$, on sait que les séries $h(z)$, $h(z^2)$, $h(z^q)$ et $h(z^{2q})$ sont algébriquement indépendantes sur $K$. En particulier, la fonction $f(z)$ est transcendante sur $K$.

On a $f(\alpha)+f(\alpha^q)=0$ d'après~\eqref{eq:dsexiii}, tandis que les fonctions $f(z)$ et $f(z^q)$ ne sont liées par aucune relation linéaire (homogène) sur $K$ (en effet on a $f(z^q)/f(z) = (1-z/\alpha)(1-z/\alpha^q)f(z)^{-1}-1 \notin K$). Ceci est toutefois compatible avec le corollaire~\ref{cor:indlin} car la matrice fondamentale de solutions~\eqref{eq:matfondsiii} a des coefficients qui ne sont pas analytiques en $\alpha$. Mais, conformément au corollaire~\ref{cor:indalgi} il existe bien une relation algébrique (inhomogène) de degré $1$ liant les fonctions $f(z)$ et $f(z^q)$ sur $K$, à savoir~\eqref{eq:dsexiii}, qui se spécialise en la relation $f(\alpha)+f(\alpha^q)=0$.
\medskip

2) Considérons le système d'équations fonctionnelles de matrice
\begin{equation}\label{eq:dsexedeux}
\begin{pmatrix}1 &0 &0\\ 0 &z+z^2+z^3+z^6 &1-(z^3+z^6)^2\\ 0 &1 &0
\end{pmatrix}
\enspace,\end{equation}
pour la transformation $z\mapsto z^3$. Une matrice fondamentale de solutions s'écrit
$$\begin{pmatrix}1 &0 &0\\ 0 &\frac{1}{1-z} &f(z)\\ 0 &\frac{1}{1-z^3} &f(z^3)
\end{pmatrix}$$
avec $f(z)$ transcendante sur $K$. En désignant par $\SS N$ l'ensemble des entiers positifs dont le nombre de chiffres non nuls dans l'écriture en base $3$ est pair, on vérifie que $f(z)=1+\sum_{i\in\SS N} z^i$, qui converge dans le disque unité sans bord, convient.

Le groupe de Galois est $\RQ G_a$, sa représentation sur l'espace des solutions est donnée par $\lambda\mapsto \left(\begin{smallmatrix} \scriptstyle1 &\scriptstyle0 &\scriptstyle2\lambda\\[1pt] \scriptstyle0 &\scriptstyle1 &\scriptstyle\lambda\\[1pt] \scriptstyle0 &\scriptstyle0 &\scriptstyle1 \end{smallmatrix}\right)$, car les composantes du vecteur $\left(\begin{smallmatrix}\scriptstyle1\\[1pt]\scriptstyle f(z)\\[1pt] \scriptstyle f(z^3)\end{smallmatrix}\right)$ sont liées sur $K$ par la relation $(z^2+z-1)f(z^3)+f(z)-\frac{z+z^2}{1-z^3}=0$. On en déduit en particulier l'évaluation $f((\sqrt{5}-1)/2)=1/(3-\sqrt{5})$, tandis que les fonctions $1$ et $f(z)$ sont linéairement indépendantes sur $K$. Enfin, $(\sqrt{5}-1)/2$ n'est pas racine du déterminant $(z^6+z^3)^2-1$ de la matrice~\eqref{eq:dsexedeux} et n'est donc pas singularité du système correspondant.
\end{exes}

Nous pouvons également déduire directement du corollaire~\ref{cor:indalgi} le théorème~\ref{thm:indlinaut}. Le corollaire suivant servira à nouveau pour démontrer le théorème~\ref{thm:indlinautsing} aux arguments singuliers.

\begin{cor}\label{cor:thmindlinautnondeg}
Soit $\RQ k_0\subset\RQ C$ un corps de nombres, $\alpha\in\RQ k_0$, $0<|\alpha|<\rho$, et $1\le\ell\le N$. Supposons que la matrice $A(z)$ du système~\eqref{eq:systintro} soit à coefficients dans $\RQ k_0(z)$ et le vecteur $\QQ f(z)$ à composantes $f_1(z),\dots,f_N(z)$ dans $\RQ k_0[[z]]$. Si $\alpha$ n'appartient pas à l'ensemble singulier de~\eqref{eq:systintro} et si les nombres $f_{1}(\alpha),\dots,f_{\ell}(\alpha)$ sont linéairement indépendants sur $\RQ k_0$, alors ils sont linéairement indépendants sur $\OO{\RQ Q}$.
\end{cor}
\begin{proof}
Raisonnons par l'absurde, soit une relation $\sum_{i=1}^\ell\lambda_i f_i(\alpha)=0$ avec $\lambda_i\in\OO{\RQ Q}$ non tous nuls. D'après le corollaire~\ref{cor:indalgi}, $\alpha$ n'appartenant pas à l'ensemble singulier de~\eqref{eq:systintro}, il existe une relation $c_0(z)+\sum_{i=1}^Nc_{i}(z)f_i(z)=0$ avec $c_{i}(z)\in\OO{\RQ Q}[z]$ satisfaisant $c_{i}(\alpha)=\lambda_i$ pour $i=1,\dots,\ell$ et $c_0(\alpha)=0$, $c_{i}(\alpha)=0$, $i=\ell+1,\dots,N$. Soit $\xi_1,\dots,\xi_s$ une base du \pre{\RQ k_0}espace vectoriel engendré dans $\OO{\RQ Q}$ par les coefficients des polynômes $c_i(z)$. On peut donc écrire $c_{i}(z)=\sum_{j=1}^sc_{i,j}(z)\xi_j$ avec $c_{i,j}(z)\in \RQ k_0[z]$, puis
$$\sum_{j=1}^s\Big(c_{0,j}(z)+\sum_{i=1}^N c_{i,j}(z)f_i(z)\Big)\xi_j=0
\enspace.$$
Et comme les séries $f_i(z)$ sont à coefficients dans $\RQ k_0$, cela entraîne
\begin{equation}\label{eq:dsdemcornondeg}
c_{0,j}(z)+\sum_{i=1}^Nc_{i,j}(z)f_i(z)=0
\enspace.\end{equation}
pour $j=1,\dots,s$, et donc $c_{0,j}(\alpha)+\sum_{i=1}^Nc_{i,j}(\alpha)f_i(\alpha)=0$, $j=1,\dots,s$. Comme $\alpha\in\RQ k_0$, $c_{i,j}(\alpha)\in\RQ k_0$. On observe que $c_{i}(\alpha)=0$ si et seulement si $c_{i,j}(\alpha)=0$ pour tout $j=1,\dots,s$. Mais $\lambda_i=c_{i}(\alpha)=\sum_{j=1}^sc_{i,j}(\alpha)\xi_j$ pour $i=1,\dots,\ell$ ne sont pas tous nuls. Ainsi, pour au moins un $j\in\{1,\dots,s\}$ on obtient une relation non triviale $\sum_{i=1}^\ell c_{i,j}(\alpha)f_i(\alpha)=0$ contredisant l'indépendance linéaire sur $\RQ k_0$ des nombres $f_1(\alpha),\dots,f_\ell(\alpha)$.
\end{proof}

\medskip

Lorsque $\alpha$ est singulier, nous nous ramenons à un système nettoyé de cette singularité. Comme déjà indiqué dans l'introduction B.Adamczewski et C.Faverjon ont obtenu un énoncé jouant le même rôle que le lemme~\ref{lem:desingaut} ci-dessous. Les approches adoptées pour établir ces résultats diffèrent toutefois sensiblement. Adamczewski et Faverjon utilisent une méthode de dédoublement itérée qui augmente la taille du système d'équations fonctionnelles mais a l'avantage de conserver les fonctions originales inaltérées, tandis que nous construisons ici un système de même taille que le système original, par un procédé de repoussage de l'ensemble singulier du système d'équations fonctionnelles qui modifie le vecteur solution considéré. De plus, notre approche ne fonctionne que moyennant une hypothèse d'indépendance linéaire des valeurs des fonctions, alors que celle de Adamczewski et Faverjon évite cet écueil.

\medskip

On suppose dorénavant que la matrice du système~\eqref{eq:systintro} est à coefficients dans $\RQ k_0[z]$ pour un certain corps de nombres $\RQ k_0\subset\RQ C$. De sorte que l'ensemble singulier du système~\eqref{eq:systintro} se réduit aux zéros du déterminant de $A(z)$ et à leurs racines \pre{q^\ell}ième, $\ell\in\RQ N$. De plus, les coefficients de $A(z)$ n'ayant pas de pôle il en est de même des composantes dans $\RQ k_0[[z]]$ du vecteur solution $\QQ f(z)$ de~\eqref{eq:systintro}, dans le disque unité sans bord de $\RQ C$.

\begin{lem}\label{lem:desingaut}
Soit $\RQ k_0\subset\RQ C$ un corps de nombres, $\alpha\in\RQ k_0$, $0<|\alpha|<1$. Supposons que la matrice $A(z)$ du système~\eqref{eq:systintro} soit à coefficients dans $\RQ k_0[z]$ et soit $\QQ f(z)$ un vecteur solution de~\eqref{eq:systintro} à composantes dans $\RQ k_0[[z]]$. On suppose que les nombres $f_1(\alpha),\dots,f_\ell(\alpha)$ sont linéairement indépendants sur $\RQ k_0$ pour un certain $\ell\in\{1,\dots,N\}$. 

Il existe alors une matrice $\tilde A(z)$ à coefficients dans $\RQ k_0[z]$ et un vecteur solution $\tilde{\QQ f}(z)$, à composantes dans $\RQ k_0[[z]]$, du système d'équations fonctionnelles correspondant tels que $\det(\tilde A(\alpha^{q^h}))\not=0$ pour tout $h\in\RQ N$ et pour $i=1,\dots,\ell$:
$$\RQ k_0f_1(\alpha)+\dots+\RQ k_0f_i(\alpha) = \RQ k_0\tilde f_1(\alpha)+\dots+\RQ k_0\tilde f_i(\alpha)\enspace.$$
\end{lem}
\begin{proof}
Soit $\beta\in\RQ k_0$ appartenant à l'ensemble singulier de~\eqref{eq:systintro}, c'est-à-dire tel que $\RR{ord}_\beta\det(A(z))>0$. Nous allons construire un nouveau système dont le déterminant de la matrice sera égal à $\det(A(z))\frac{z^q-\beta}{z-\beta}$. Comme $\det(A(\beta))=0$, il existe au moins une relation non triviale entre les lignes de cette matrice. Soit $j$ le plus petit entier dans $\{1,\dots,N\}$ tel que les $j$ premières lignes de $A(\beta)$ soient liées et notons $\lambda_{1},\dots,\lambda_{j}\in \RQ k_0$, $\lambda_j=1$, les coefficients d'une relation entre ces lignes: $\sum_{k=1}^{j}\lambda_ka_{k,h}(\beta)=0$, pour tout $h=1,\dots,N$. On pose $S(z)$ la matrice dont toutes les lignes coïncident avec la matrice identité $N\times N$, sauf la \pre{j}ième qui s'écrit
$$\Big(\frac{\lambda_{1}}{z-\beta}, \dots, \frac{\lambda_{j}}{z-\beta}, 0, \dots, 0\Big)
\enspace.$$
Cette matrice $S(z)$ est triangulaire inférieure de déterminant $\frac{1}{z-\beta}$. On vérifie que le produit $S(z)A(z)S(z^q)^{-1}$ est à coefficients dans $\RQ k_0[z]$, car $z-\beta$ divise tous les polynômes $\sum_{k=1}^{j}\lambda_ka_{k,h}(z)$, $h=1,\dots,N$. Son déterminant est égal à $\det(A(z))\frac{z^q-\beta}{z-\beta}$. De plus, $S(z)\QQ f(z)$ est solution du système d'équations fonctionnelles associé à $S(z)A(z)S(z^q)^{-1}$ et a les mêmes composantes que $\QQ f(z)$ sauf la \pre{j}ième qui s'écrit $\sum_{k=1}^{j}\lambda_kf_k(z)/(z-\beta)\in\RQ k_0[[z]]$. Si $\beta\not=\alpha$ ou $j>\ell$ le bloc des coefficients d'indices $\le\ell$ de $S(\beta)$ et une matrice $\ell\times\ell$ triangulaire inférieure inversible à coefficients dans $\RQ k_0$, ce qui assure que pour tout $i=1,\dots,\ell$ les \pre{\RQ k_0}espaces vectoriels engendrés par les $i$ premières composantes de $S(\alpha)\QQ f(\alpha)$ d'une part et de $\QQ f(\alpha)$ d'autre part, coïncident. 

Soit maintenant $\alpha\in\RQ k_0$ comme dans l'énoncé et $m\in\RQ N$ le plus grand entier tel que $\alpha^{q^m}$ soit racine de $\det(A(z))$ (si $m$ n'existe pas l'énoncé est clair avec $\tilde A(z)=A(z)$ et $\tilde{\QQ f}(z)=\QQ f(z)$). Définissons une suite d'entiers $n_0,\dots,n_m$ de la façon suivante: $n_m=\RR{ord}_{\alpha^{q^m}}\det(A(z))$, puis $n_i=n_{i+1}+\RR{ord}_{\alpha^{q^i}}\det(A(z))$, pour $i=m-1,\dots,0$. Appliquant la construction précédente $n_m$ fois avec $\beta=\alpha^{q^m}$, puis $n_{m-1}$ fois avec $\beta=\alpha^{q^{n_{m-1}}}$ fois et ainsi de suite jusqu'à $n_0$ fois avec $\beta=\alpha$, on obtient une matrice $\tilde S(z)$ triangulaire inférieure (produit de toutes les matrices introduites successivement au cours de la procédure) telle que la matrice $\tilde A(z)=\tilde S(z)A(z)\tilde S(z)^{-1}$ soit à coefficients dans $\RQ k_0[z]$, de déterminant ne s'annulant pas en $\alpha^{q^h}$, $h\in\RQ N$, et un vecteur solution $\tilde{\QQ f}(z)=\tilde S(z)\QQ f(z)\in\RQ k_0[[z]]$. On observe qu'à chaque pas le nombre de zéros du déterminant (comptés avec multiplicité) dans $\{\alpha^{q^h},h\in\RQ N\}$ est fini et n'augmente pas, tandis que le nombre de zéros dans ce même ensemble, de valeur absolue minimale, diminue strictement. On remarque aussi que lorsqu'on applique la construction expliquée ci-dessus avec $\beta=\alpha$, l'hypothèse d'indépendance linéaire sur $\RQ k_0$ des valeurs en $\alpha$ des $\ell$ première composantes du vecteur solution assure que les $\ell$ premières lignes de la matrice du système sont indépendantes et donc $j>\ell$. Ceci entraîne que, pour $i=1,\dots,\ell$, les \pre{\RQ k_0}espaces vectoriels engendrés par les $i$ premières composantes de $\tilde{\QQ f}(\alpha)=\tilde S(\alpha)\QQ f(\alpha)$ d'une part et de $\QQ f(\alpha)$ d'autre part, coïncident.
\end{proof}

L'énoncé suivant établit le théorème~\ref{thm:indlinautsing} de l'introduction dans toute sa généralité, en combinant la réduction du lemme~\ref{lem:desingaut} et le corollaire~\ref{cor:thmindlinautnondeg}.

\begin{cor}\label{cor:corlinindaut}
Sous les hypothèses du théorème~\ref{thm:indlinautsing} la matrice $A(z)$ est à coefficients dans $\RQ k_0[z]$. Alors, si les nombres $f_1(\alpha),\dots,f_\ell(\alpha)$ sont linéairement indépendants sur $\RQ k_0$, ils sont linéairement indépendants sur $\OO{\RQ Q}$.
\end{cor}
\begin{proof}
Les hypothèses du théorème~\ref{thm:indlinautsing} valident celles des corollaire~\ref{cor:thmindlinautnondeg} et lemme~\ref{lem:desingaut}. Le lemme~\ref{lem:desingaut} fournit un système d'équations fonctionnelles dont $\alpha$ n'est pas singularité et un vecteur solution $\tilde{\QQ f}(z)$ dont les valeurs en $\alpha$ des $\ell$ premières composantes sont linéairement indépendantes sur $\RQ k_0$, car
\begin{equation}\label{eq:dspreuvecorlinindaut}
\RQ k_0\tilde f_1(\alpha)+\dots+\RQ k_0\tilde f_\ell(\alpha) = \RQ k_0f_1(\alpha)+\dots+\RQ k_0f_\ell(\alpha)
\end{equation}
ont même dimension $\ell$ sur $\RQ k_0$. En particulier, le vecteur $\tilde{\QQ f}(\alpha)$ satisfait les hypothèses du corollaire~\ref{cor:thmindlinautnondeg} et on en conclut que $\OO{\RQ Q}\tilde f_1(\alpha)+\dots+\OO{\RQ Q}\tilde f_\ell(\alpha)$ a dimension $\ell$ sur $\OO{\RQ Q}$, d'où le résultat car il suit de~\eqref{eq:dspreuvecorlinindaut}
$$\OO{\RQ Q}\tilde f_1(\alpha)+\dots+\OO{\RQ Q}\tilde f_\ell(\alpha) = \OO{\RQ Q}f_1(\alpha)+\dots+\OO{\RQ Q}f_\ell(\alpha)
\enspace.$$
\end{proof}

\section{Convergence}\label{sec:convergence}
Nous montrons qu'il existe une matrice fondamentale de solutions de~\eqref{eq:systintro} à coefficients séries en $z$ si et seulement si $A(0)=\RR{Id}$. Et que dans ce cas, toute matrice de solutions est méromorphe dans le disque unité et l'une d'entre elles satisfait $U(0)=\RR{Id}$.

Commençons par un lemme sur les solutions de~\eqref{eq:systintro} à composantes séries formelles de Laurent, \emph{i.e.} dans $\RQ C((z))$. Soit $S$ l'ensemble des pôles des coefficients de $A(z)$ distincts de $0$, comptés avec multiplicité, et $m(q-1)$ le plus petit multiple de $q-1$ majorant les multiplicités des pôles des coefficients de $A(z)$ en $0$, on pose $\Delta_A(z)=z^{m(q-1)}\prod_{\xi\in S}(1-\xi^{-1}z)$.

\begin{lem}\label{lem:convergence}
Soit $\QQ f(z)\in\RQ C((z))^N$ satisfaisant~\eqref{eq:systintro}, alors les composantes de $\QQ f(z)$ définissent des fonctions méromorphes dans le disque unité sans bord de $\RQ C$, ayant pour dénominateur commun dans ce disque privé de l'origine la fonction analytique définie par le produit $\prod_{\xi\in S,\ell\in\RQ N}(1-\xi^{-1}z^{q^\ell})$.
\end{lem}
\begin{proof}
Montrons d'abord le résultat lorsque $A(z)$ est à coefficients dans $\RQ C[z]$. \'Ecrivons $\QQ f(z) = \sum_{h\ge-n}\QQ c_{h}z^h$ avec $\QQ c_h\in\RQ C^N$, $n\in\RQ N$, et $A(z)=\sum_{j=0}^dA_jz^j$ avec $A_j$ matrice $N\times N$ à coefficients dans $\RQ C$. Notons $M_h$ le maximum des normes $L^1$ des vecteurs $\QQ c_{-n},\dots,\QQ c_h$. On déduit du système~\eqref{eq:systintro} les équations
$$\QQ c_h = \sum_{\stackrel{0\le j\le \min(h+n,d)}{\scriptscriptstyle j\equiv h(q)}}A_{j}\QQ c_{(h-j)/q}
\enspace.$$
De sorte que $M_h\le \Vert A\Vert M_{[h/q]}$  où $\Vert A\Vert$ désigne la somme des normes $L^1$ des lignes des matrices $A_0,\dots,A_d$. Ainsi a-t-on
$$M_h \le \Vert A\Vert^{[\log(h)/\log(q)]}M_{q-1} \le h^{\log\Vert A\Vert/\log(q)}M_{q-1}$$
pour tout $h\ge q$, ce qui entraîne que les séries composantes de $\QQ f(z)$ con\-ver\-gent dans le disque unité sans bord de $\RQ C$ (éventuellement privé de l'origine si $n>0$).

Dans le cas général, on applique ce qui précède aux deux systèmes d'équations fonctionnelles associés aux matrices $\Delta_A(z)$ et $\Delta_A(z)A(z)$ à coefficients dans $\RQ C[z]$. Une solution du premier système est donnée par le produit $\delta(z)=z^{-m}\prod_{\xi\in S,\ell\in\RQ N}(1-\xi^{-1}z^{q^\ell})$ qui se développe en un élément de $\RQ C((z))$. La seconde équation a pour solution $\delta(z)\QQ f(z)$ qui est également à coefficients dans $\RQ C((z))$. Par la première partie de la preuve, toutes ces séries convergent dans le disque unité sans bord de $\RQ C$ privé de l'origine. En l'origine elles définissent des fonctions méromorphes et on conclut donc que le quotient $\QQ f(z)=\delta(z)\QQ f(z)/\delta(z)$ est un vecteur à composantes méromorphes dans le disque unité sans bord de $\RQ C$. De plus, $\QQ f(z)\prod_{\xi\in S,\ell\in\RQ N}(1-\xi^{-1}z^{q^\ell}) = z^m\delta(z)\QQ f(z)$ est analytique dans le disque unité sans bord privé de l'origine.
\end{proof}

\begin{prop}\label{prop:enonceconvergence}
Soit $\SS O$ un sous-anneau de $\RQ C$ et $A(z)$ une matrice à coefficients dans $\RQ C(z)$ satisfaisant $A(0)=\RR{Id}$, telle que $\Delta_A(z)\in\SS O[z]$ et $\Delta_A(z)A(z)$ soit à coefficients dans $\SS O[z]$. Alors le produit infini
$$A(z)A(z^q)\dots A(z^{q^\ell})\dots$$
converge \pre{z}adiquement vers une matrice $U(z)$ à coefficients dans $\SS O[[z]]$, méromorphes dans le disque unité sans bord de $\RQ C$. De plus, $U(z)$ satisfait $U(0)=\RR{Id}$ et le système d'équations fonctionnelles~\eqref{eq:systintro}.
 
Un dénominateur commun des coefficients de $U(z)$ est la fonction, analytique dans le disque unité sans bord de $\RQ C$, $\prod_{\xi\in S,\ell\in\RQ N}(1-\xi^{-1}z^{q^\ell})$.
\end{prop}
\begin{undef}{Nota Bene}
La convergence \pre{z}adique du produit infini est à comprendre après développement des coefficients de la matrice $A(z)$ en séries de $\SS O[[z]]$, lorsque ceux-ci ne sont pas des polynômes.
\end{undef}

\begin{proof}
Appelons $A^{(\ell)}(z)$ la matrice obtenue en tronquant le produit infini à l'exposant $q^{\ell-1}$. Comme $A(0)=\RR{Id}$, on vérifie que les coefficients dans $\SS O[[z]]$ des matrices $A^{(\ell')}(z)$ pour $\ell'\geq\ell$ coïncident modulo $z^{q^\ell}$, ce qui montre que le produit converge bien \pre{z}adiquement vers une matrice à coefficients dans $\SS O[[z]]$. 

La méromorphie des coefficients de $U(z)$ dans le disque unité sans bord de $\RQ C$ suit alors du lemme~\ref{lem:convergence}. Et, de par sa définition, il est clair que la matrice $U(z)$ satisfait $U(0)=\RR{Id}$ et le système d'équations fonctionnelles $U(z)=A(z)U(z^q)$.
\end{proof}

\begin{cor}\label{cor:enonceconvergence}
Supposons que $0$ n'est pas pôle du déterminant de $A(z)$ et soit $\SS O$ le sous-anneau de $\RQ C$ engendré par les coefficients des coefficients de $\Delta_A(z)$ et $\Delta_A(z)A(z)$.

Il existe une matrice fondamentale de solutions $U(z)$ de~\eqref{eq:systintro} dont les coefficients sont analytiques dans un voisinage de l'origine de $\RQ C$ si et seulement si $A(0)=\RR{Id}$.

De plus, dans ce cas les coefficients de $U(z)$ sont méromorphes dans le disque unité sans bord. La matrice $U(z)$ est unique à multiplication à droite par une matrice à coefficients dans $\RQ C$ près, qui peut être choisie de sorte que $U(0)=\RR{Id}$ et $U(z)$ soit à coefficients dans $\SS O[[z]]$. Un dénominateur commun dans $\SS O[[z]]$ des coefficients de $U(z)$ est $\prod_{\xi\in S,\ell\in\RQ N}(1-\xi^{-1}z^{q^\ell})$.
\end{cor}
\begin{proof}
Dans une direction, si $A(0)=\RR{Id}$ l'énoncé suit de la proposition~\ref{prop:enonceconvergence}, car la proposition~\ref{wronskien} montre que les matrices fondamentales de solutions se déduisent l'une de l'autre par multiplication à droite par une matrice inversible à coefficients dans $\RQ C$.

Dans l'autre direction, l'équation~\eqref{eq:systintro} entraîne
\begin{equation}\label{eq:det}
{\det(U(z))} = \det(A(z))\cdot{\det(U(z^q))}
\end{equation}
qui montre
\begin{equation}\label{eq:orddet}
0=\RR{ord}_0(\det(A(z)))+(q-1).\RR{ord}_0(\det(U(z)))
\enspace,\end{equation}
car $\RR{ord}_0(\det(U(z^q)))=q.\RR{ord}_0(\det(U(z)))$. Supposons que~\eqref{eq:systintro} a une matrice fondamentale de solutions à coefficients dans $\RQ C[[z]]$, analytique dans un voisinage de l'origine. Alors, comme $q\ge 2$, $\det(U(z))\in\RQ C[[z]]$ et $0$ n'est pas pôle de $\det(A(z))$, \eqref{eq:orddet} implique
$$\RR{ord}_0(\det(A(z)))=\RR{ord}_0(\det(U(z)))=0
\enspace.$$
Ainsi la matrice $U(0)$ est inversible et il en est de même de $U(z)$ pour tout $z$ suffisamment proche de $0$. L'identité $A(z)=U(z)U(z^q)^{-1}$ entraîne $A(0)=U(0)U(0)^{-1}=\RR{Id}$. Le reste de l'énoncé découle à nouveau de la proposition~\ref{prop:enonceconvergence}.
\end{proof}

\medskip

%\begin{acknowledgements}\label{ackref}
Je remercie Daniel Bertrand pour son aide \emph{via} de nombreuses discussions sur la théorie de Galois des équations fonctionnelles. Je remercie également Boris Adamczewski pour m'avoir fait part de son travail~\cite{AdFa=15} avec Colin Faverjon.
%\end{acknowledgements}

\end{document}